\documentclass[12pt]{amsart}
\usepackage[all]{xy}
\usepackage[latin1]{inputenc}
\usepackage{hyperref}
\usepackage{amssymb}
\usepackage{tensor}
\newcommand{\C}{{\mathbb C} }
\newcommand{\R}{{\mathbb R} }
\newcommand{\cA}{{\mathcal A} }

\newcommand{\cE}{{\mathcal E} }

\newcommand{\cL}{{\mathcal L} }
\newcommand{\cM}{{\mathcal M} }

\newcommand{\cO}{{\mathcal O} }

\newcommand{\cT}{{\mathcal T} }

\newcommand{\cX}{{\mathcal X} }

\newcommand{\cZ}{{\mathcal Z} }
\newcommand{\cH}{{\mathcal H} }
\newcommand{\cK}{{\mathcal K} }
\newcommand{\wh}{\widehat}
\newcommand{\wt}{\widetilde}
\newcommand{\pt}{\partial}
\def\ol#1{{\overline{#1}}}
\def\funo#1{\footnotesize${#1}$\normalsize}
\newtheorem*{Maintheorem*}{Main Theorem}
\newtheorem*{theorem*}{Theorem}
\newtheorem{maintheorem}{Theorem}

\newtheorem{theorem}{Theorem}
\newtheorem{definition}{Definition}
\newtheorem{lemma}{Lemma}
\newtheorem{remark}{Remark}
\newtheorem{proposition}{Proposition}
\newtheorem{corollary}{Corollary}

\def\lsc#1#2{\tensor[^{#1}]{#2}{}}
\def\ke{K{\"a}h\-ler-Ein\-stein }
\def\ks{Ko\-dai\-ra-Spen\-cer }
\def\ka{K{\"a}h\-ler }
\def\wp{Weil-Pe\-ters\-son }
\def\tei{Teich\-mül\-ler }
\def\fs{Fubini-Study }
\def\ii{\sqrt{-1}}
\def\ddb{\sqrt{-1}\partial\overline{\partial}}
\def\C{\mathbb{C}}
\def\cinf{C^\infty}
\def\gab{{g_{\alpha\ol\beta}}}
\def\psh{plurisubharmonic}
\def\RP{$R^{n-p}f_*\Omega^p_{\cX/S}(\cK_{\cX/S}^{\otimes m})$\ }
\def\we{\wedge}

\begin{document}

\title[Positivity of relative canonical bundles]{Positivity of relative canonical
bundles and applications}

\dedicatory{Dedicated to the memory of Eckart Viehweg}

\author{Georg Schumacher}
\address{Fachbereich Mathematik und Informatik,
Philipps-Universit\"at Marburg, Lahnberge, Hans-Meerwein-Straße, D-35032
Marburg,Germany}
\email{schumac@mathematik.uni-marburg.de}
\date{}



\maketitle

\begin{abstract}
Given a family $f:\cX \to S$ of canonically polarized manifolds, the
unique \ke metrics on the fibers induce a hermitian metric on the relative
canonical bundle $\cK_{\cX/S}$. We use a global elliptic equation to show
that this metric is strictly positive on $\cX$, unless the family is
infinitesimally trivial.

For degenerating families we show that the curvature form on the total
space can be extended as a (semi-)positive closed current. By fiber
integration it follows that the generalized \wp form on the base possesses
an extension as a positive current. We prove an extension theorem for
hermitian line bundles, whose curvature forms have this property. This
theorem can be applied to a determinant line bundle associated to the
relative canonical bundle on the total space. As an application the
quasi-projectivity of the moduli space $\cM_{\text{can}}$ of canonically
polarized varieties follows.

The direct images $R^{n-p}f_*\Omega^p_{\cX/S}(\cK_{\cX/S}^{\otimes m})$,
$m > 0$, carry natural hermitian metrics. We prove an explicit formula for
the curvature tensor of these direct images.  We apply it to the morphisms
$S^p\cT_S \to R^pf_*\Lambda^p\cT_{\cX/S}$ that are induced by the \ks map
and obtain a differential geometric proof for hyperbolicity properties of
$\cM_{\text{can}}$.
\end{abstract}

\tableofcontents

\section{Introduction}
For any holomorphic family $f: \cX \to S$ of canonically polarized,
complex manifolds, the unique \ke metrics on the fibers define an
intrinsic metric on the relative canonical bundle $\cK_{\cX/S}$, whose
curvature form has at least as many positive eigenvalues as the dimension
of the fibers indicates. The construction is functorial in the sense of
compatibility with base changes.

\begin{Maintheorem*}\label{th:main1}
Let $\cX \to S$ be a holomorphic family of canonically polarized, compact,
complex manifolds, which is nowhere infinitesimally trivial. Then the
curvature of the hermitian metric on $\cK_{\cX/S}$ that is induced by the
\ke metrics on the fibers is strictly positive.
\end{Maintheorem*}

Actually the first variation of the metric tensor in a family of compact
\ke manifolds contains the information about the corresponding
deformation, more precisely, it contains the harmonic representatives of
the \ks classes. The positivity of the hermitian metric will be measured
in terms of a certain global function. Essential is an elliptic equation
on the fibers, which relates this function to the pointwise norm of the
harmonic \ks forms. The strict positivity of the corresponding (fiberwise)
operator $(\Box + id)^{-1}$, where $\Box$ is the complex Laplacian, can be
seen in a direct way (cf. \cite{sch-preprint}).  For families of compact
Riemann surfaces this operator had been considered in the context of
automorphic forms by Wolpert \cite{wo}. Later in higher dimensions a
similar equation arose in the work of Siu \cite{siu:canlift} for families
of canonical polarized manifolds.

Here we reduce estimates for the positivity of the curvature of
$\cK_{\cX/S}$ on $\cX$ to estimates of the resolvent kernel of the above
integral operator, whose positivity was already shown by Yosida in
\cite{yos}. Finally estimates for the resolvent kernel follow from the
estimates for the heat kernel, which were achieved by Cheeger and Yau in
\cite{c-y}.

The positivity of the relative canonical bundle and the methods involved
are closely related to different questions. We will treat the construction
of a positive line bundle on the moduli space of canonically polarized
manifolds proving its quasi-projectivity, and the question about its
hyperbolicity.

\begin{maintheorem}\label{th:main2.0}
Let $(\cK_{\cX/\cH},h)$ be the relative canonical bundle on the total
space over the Hilbert scheme, equipped with the hermitian metric that is
induced by the \ke metrics on the fibers. Then the curvature form
$\omega_\cX$ extends to the total space $\ol\cX$ over the compact Hilbert
scheme $\ol\cH$ as a positive, closed current $\omega_{\ol \cX}$ with at
most analytic singularities.
\end{maintheorem}

The proof depends on Yau's $C^0$-estimates \cite{calyau} that he used for
the construction of \ke metrics.

We consider a certain determinant line bundle of the relative canonical
bundle, which is defined in a functorial way on the base of any family of
canonically polarized varieties. In particular, it exists on the open part
of the Hilbert scheme. It carries a Quillen metric, and its curvature form
is the generalized \wp form. The latter is equal to the fiber integral
$\int_{\cX/\cH} c_1(\cK_{\cX/\cH},h)^{n+1}$, where $n$ is the fiber
dimension. These quantities descend to the moduli space of canonically
polarized manifolds. Again we see that the curvature form extends as a
closed, (semi-)positive current to a compactification of the moduli space,
which is known to exist as a Moishezon space according to Artin's theorem.

Based upon Siu's decomposition theorem for positive closed currents we
prove the following theorem.

\begin{maintheorem}\label{th:main2}
Let $Y$ be a normal space and $Y' \subset Y$ the complement of a closed
analytic, nowhere dense subset. Let $L'$ be a holomorphic line bundle on
$Y'$ together with a hermitian metric $h'$ of semi-positive curvature,
which may be singular. Assume that the curvature current can be extended
to $Y$ as a positive, closed current $\omega$. Then there exists a
holomorphic line bundle $(L,h)$ with a singular hermitian metric of
semi-positive curvature, whose restriction to $Y'$ is isomorphic to
$(L',h')$. The metric $h$ can be chosen with at most analytic
singularities, if $\omega$ has this property.
\end{maintheorem}

This proves the extension of the determinant line bundle and the Quillen
metric to a compactification of the normalization. Finally the complex
structure at points of the compactifying divisor has to be changed.

\begin{maintheorem}
The generalized \wp form on the moduli stack of canonically polarized
varieties is strictly positive. A multiple is the Chern form of a
determinant line bundle, equipped with a Quillen metric. This line bundle
extends to a compactification of the moduli space, and the Quillen metric
extends as a (semi-)positive singular hermitian metric with at most
analytic singularities (in the orbifold sense).
\end{maintheorem}

These facts imply a short proof for the quasi-projectivity of the moduli
space of canonically polarized manifolds (cf.\ Section~\ref{se:comp}). It
was pointed out by Eckart Viehweg in \cite{v2} that the extension of the
\wp current is not automatic. This issue is closely related to the
extendability of the determinant line bundle, a question emphasized by
Kollar \cite{ko1} -- these problems are addressed in our manuscript.

We note that our approach uses the functoriality of the construction of a
determinant line bundle, equipped with a Quillen metric, which was
generalized in \cite{f-s:extremal} to smooth families over singular base
spaces. Once this generalization is established, when dealing with
singular fibers, we can use desingularizations of  the base spaces,
allowing for non-reduced fibers, and prove extension theorems for line
bundles and curvature currents.

Responding to a remark by Robert Berman we mention that the theory of
non-pluripolar products of globally defined currents of Boucksom,
Eyssidieux, Guedj, and Zeriahi from \cite{begz} applies to the \wp volume
of a component  $\cM$ of the moduli space of canonically polarized
manifolds:
\begin{corollary}\label{co:finwp}
Let $\omega^{WP}$ be the \wp form on the (open) moduli space $\cM$, which
is of class $\cinf$ in the orbifold sense. Then
$$
\int_\cM (\omega^{WP})^{\dim \cM} < \infty.
$$
\end{corollary}

\medskip

The {\em second application} concerns the direct image sheaves
$$
R^{n-p}f_*\Omega^p_{\cX/S}(\cK_{\cX/S}^{\otimes m}).
$$
These are equipped with a natural hermitian metric that is induced by the
$L^2$-inner product of harmonic tensors on the fibers of $f$. In the
second part of this article we will give an explicit formula for the
curvature tensor. Estimates will be related to the above discussion of the
resolvent kernel.

A main motivation of our approach is the study of the curvature of the
classical \wp metric, in particular the computation of the curvature
tensor by Wolpert \cite{wo} and Tromba \cite{tr}. It immediately implies
the hyperbolicity of the classical \tei space. For families of higher
dimensional manifolds with metrics of constant Ricci curvature the
generalized \wp metric  was explicitly  computed by Siu in
\cite{siu:canlift}. In \cite{sch:curv} a formula in terms of elliptic
operators and harmonic \ks tensors was derived. However, the curvature of
the generalized \wp metric seems not to satisfy any negativity condition.
This difficulty was overcome in the work of Viehweg and Zuo in \cite{v-z}.
Their approach to hyperbolicity makes use of a period map.

On the other hand our results are motivated by Berndtsson's result on the
Nakano-positivity for certain direct images. In \cite{sch-preprint} we
showed that the positivity of $f_*\cK_{\cX/S}^{\otimes 2}$ together with
the curvature of the generalized \wp metric is sufficient to imply a
hyperbolicity result for moduli of canonically polarized complex manifolds
in dimension two.

For ample $K_X$, and a fixed number $m\geq 1$, e.g.\ $m=1$, the cohomology
groups $H^{n-p}(X,\Omega^p_X(K^{\otimes m}_X))$ are critical with respect
to the Kodaira-Nakano vanishing theorem. The understanding of this
situation is the other main motivation. We will consider the relative
case. Let $A_{i \ol\beta}^\alpha(z,s) \pt_\alpha dz^{\ol\beta}$ be a
harmonic \ks form. Then for  $s \in S$ the cup product together with the
contraction defines a mapping
\begin{eqnarray*}
A_{i \ol\beta}^\alpha \pt_\alpha dz^{\ol\beta}\cup \textvisiblespace :
\cA^{0,n-p}(\cX_s,\Omega^p_{\cX_s}(\cK_{\cX_s}^{\otimes m})) &\to&
\cA^{0,n-p+1}(\cX_s,\Omega^{p-1}_{\cX_s}(\cK_{\cX_s}^{\otimes m}))\\ A_{\ol \jmath
\alpha}^\ol\beta \pt_\ol\beta dz^{\alpha}\cup \textvisiblespace :
\cA^{0,n-p}(\cX_s,\Omega^p_{\cX_s}(\cK_{\cX_s}^{\otimes m})) &\to&
\cA^{0,n-p-1}(\cX_s,\Omega^{p+1}_{\cX_s}(\cK_{\cX_s}^{\otimes m})).
\end{eqnarray*}
\enlargethispage{1.5cm} We will apply the above products to harmonic
$(0,n-p)$-forms. In general the results are not harmonic. We denote the
pointwise $L^2$ inner product by a dot.

The computation of the curvature tensor yields the following result. For
necessary assumptions cf.\ Section~\ref{se:statement}.

\vfill

\break

\begin{maintheorem}\label{th:curvIV}
The curvature tensor for $R^{n-p}f_*\Omega^p_{\cX/S}(\cK_{\cX/S}^{\otimes
m})$ is given by
\vspace{-5mm}
\begin{eqnarray}
R_{i\ol\jmath}^{\phantom{{i\ol\jmath}}\ol\ell k}(s)&=& m \int_{\cX_s}
\left( \Box + 1 \right)^{-1}(A_i\cdot A_\ol\jmath) \cdot(\psi^k \cdot
\psi^\ol\ell) g\/ dV\nonumber\\
&& \quad + m \int_{\cX_s} \left( \Box + m \right)^{-1} (A_i\cup\psi^k)
\cdot (A_\ol\jmath \cup \psi^\ol\ell) g\/ dV \\
&& \quad + m \int_{\cX_s} \left( \Box - m \right)^{-1}
(A_i\cup\psi^\ol\ell)\cdot (A_\ol\jmath \cup \psi^k) g\/ dV.
\nonumber
\end{eqnarray}
\end{maintheorem}

The potentially negative third term is identically zero for $p=n$, i.e.\
for $f_*\cK_{\cX/S}^{\otimes(m+1)}$. From our Main Theorem, with
Theorem~\ref{th:curvIV} we immediately get a fact which also follows from
the Main Theorem with \cite[Theorem~1.2]{berndtsson}\footnote{Very
recently Berndtsson considered the curvature of $f_* (\cK_{\cX/S} \otimes
\cL)$ in \cite{berndtsson-pre}.}:
\begin{corollary}
The locally free sheaf $f_*\cK_{\cX/S}^{\otimes(m+1)}$ is Nakano-positive.
\end{corollary}
An estimate is given in Corollary~\ref{co:curv1}. It contains the
following inequality.
$$
R(A,\ol A, \psi,\ol\psi) \geq P_n(d(\cX_s)) \cdot \|A\|^2\cdot
\|\psi\|^2,
$$
where the coefficient $P_n(d(\cX_s))>0$ is an explicit function depending
on the dimension and the diameter $d(\cX_s)$ of the fibers.

For one dimensional fibers and $m=1$, we are considering the $L^2$-inner
product on the sheaf $f_*\cK_{\cX/S}^{\otimes 2}$ of quadratic holomorphic
differentials for the \tei space of Riemann surfaces of genus larger than
one, which is dual to the classical \wp metric: According to Wolpert
\cite{wo} the sectional curvature is negative, and the holomorphic
sectional curvature is bounded from above by a negative constant. A
stronger curvature property, which is related to strong rigidity, was
shown in \cite{sch:teich}.

The strongest result on curvature by Liu, Sun, and Yau \cite{yau} now
follows from Corollary~\ref{co:curvmcan} (cf.\ also
Section~\ref{se:statement}).

\begin{corollary}[\cite{yau}]\label{co:lsy}
The \wp metric on the \tei space of Riemann surfaces of genus $p>1$ is
dual Nakano negative.
\end{corollary}

\medskip

By Serre duality Theorem~\ref{th:curvIV} yields the following version,
which contains the curvature formula for the generalized \wp metric for
$p=1$. Again a tangent vector of the base is identified with a harmonic
\ks form $A_i$, and $\nu_k$ stands for a section of the relevant sheaf:
\begin{maintheorem}
The curvature for $R^pf_*\Lambda^p\cT_{\cX/S}$ equals
\begin{eqnarray}
R_{i\ol\jmath   k \ol\ell }(s)&=&- \int_{\cX_s}
\left( \Box + 1 \right)^{-1}(A_i\cdot A_\ol\jmath)
\cdot(\nu_k \cdot \nu_\ol\ell) g\/ dV\nonumber\\
&& \quad - \int_{\cX_s} \left( \Box + 1 \right)^{-1} (A_i\wedge\nu_\ol\ell)
\cdot (A_\ol\jmath \wedge \nu_k) g\/ dV \\
&& \quad -  \int_{\cX_s} \left( \Box - 1 \right)^{-1}
(A_i\wedge \nu_k)\cdot (A_\ol\jmath \wedge \nu_\ol\ell) g\/ dV.
\nonumber
\end{eqnarray}
\end{maintheorem}

Now an upper semi-continuous Finsler metric of negative holomorphic
curvature on any relatively compact subspace of the moduli stack of
canonically polarized varieties can be constructed so that any such space
is hyperbolic with respect to the orbifold structure.

We get immediately the following fact related to Shafarevich's
hyperbolicity conjecture for higher dimensions, which was solved by
Migliorini \cite{m}, Kovács \cite{kv1,kv2,kv3}, Bedulev-Viehweg
\cite{b-v}, and Viehweg-Zuo \cite{v-z,v-z2}.

{\bf Application.} {\it Let $\cX \to C$ be a non-isotrivial holomorphic
family of canonically polarized manifolds over a compact curve. Then
$g(C)>1$.}

\medskip

Similar results hold for families of polarized Ricci-flat manifolds. These
will appear elsewhere.\footnote{Dissertation in progress.}
\bigskip

\section{Fiber integration and Lie derivatives}

\subsection{Definition of fiber integrals and basic
properties}\label{sb:fibint} We denote by $\{\cX_s\}_{s\in S}$ a
holomorphic family of compact complex manifolds $\cX_s$ of dimension $n>0$
parameterized by a reduced complex space $S$. By definition, it is given
by a proper holomorphic submersion $f:\cX \to S$, such that the $\cX_s$
are the fibers $f^{-1}(s)$ for $s\in S$. In case of a smooth space $S$, if
$\eta$ is a differential form of class $\cinf$ of degree $2n+r$ the fiber
integral
$$
\int_{\cX/S} \eta
$$
is a differential form of degree $r$ on $S$. It can be defined as follows:
Fix a point $s_0\in S$ and denote by $X=\cX_{s_0}$ the fiber. Let $U
\subset S$ be an open neighborhood of $s_0$  such that there exists a
$\cinf$ trivialization of the family:
$$
\xymatrix{f^{-1}U  \ar[d]_{f|f^{-1}U } & X \ar[l]_{\Phi}^\sim\ar[dl]^{pr}\times U\\ U }
$$
Let $z=(z^1,\ldots,z^n)$ and $s=(s^1,\ldots,s^k)$ be local (holomorphic)
coordinates on $X$ and $S$ resp.

The pull-back $\Phi^*\eta$ possesses a summand $\eta'$, which is of the
form $\sum \eta_k(z,s) dV_z \wedge d\sigma^{k_1}\wedge\ldots\wedge
d\sigma^{k_r}$, where the $\sigma^\kappa$ run through the real and complex
parts of $s^j$, and where $dV_z$ denotes the relative Euclidean volume
element. Now
$$
\int_{\cX/S} \eta := \int_{X\times S/S} \Phi^*\eta := \sum_{k=(k_1,\ldots,k_r)} \left(\int_{\cX_s}
\eta_k(z,s) dV_z \right) d\sigma^{k_1}\wedge\ldots\wedge d\sigma^{k_r}.
$$
The definition is independent of the choice of coordinates and
differentiable trivializations. The fiber integral coincides with the
push-forward of the corresponding current. Hence, if  $\eta$ is a
differentiable form of type $(n+r,n+s)$, then the fiber integral is of
type $(r,s)$.

Singular base spaces are treated as follows: Using deformation theory, we
can assume that $S\subset W$ is a closed subspace of some open set $W
\subset \mathbb C^N$, and that an almost complex structure is defined on
$X\times S$ so that $\cX$ is contained in the integrable locus. Then by
definition, a differential form of class $\cinf$ on $\cX$ will be given on
the whole ambient space $X \times W$ (with a type decomposition defined on
$\cX$).

Fiber integration commutes with taking exterior derivatives:
\begin{equation}\label{eq:fibintallg}
d \int_{\cX/S} \eta = \int_{\cX/S} d\eta,
\end{equation}
and since it preserves the type (or to be seen explicitly in local
holomorphic coordinates), the same equation holds true for $\pt$ and
$\ol\pt$ instead of $d$.

A \ka form $\omega_\cX$ on a singular space, by definition is a form that
possesses locally a $\pt\ol\pt$-potential, which is the restriction of a
$\cinf$ function on a smooth ambient space. It follows from the above
facts that given a \ka form  $\omega_\cX$  on the total space, the fiber
integral
\begin{equation}
\int_{\cX/S} \omega_\cX^{n+1}
\end{equation}
is a \ka form on the base space $S$, which possesses locally a smooth
$\pt\ol\pt$-potential, even if the base space of the smooth family is
singular.

For the actual computation of exterior derivatives of fiber integrals
\eqref{eq:fibintallg}, in particular of functions, given by integrals of
$(n,n)$-forms on the fibers, the above definition seems to be less
suitable. Instead the problem is reduced to derivatives of the form
\begin{equation}\label{eq:derfibint}
\frac{\pt}{\pt s^i} \int_{\cX_s} \eta,
\end{equation}
where only the vertical components of $\eta$ contribute to the integral.
Here and later we will always use the summation convention.
\begin{lemma}\label{le:intLie}
Let
$$
w_i=\left.\left(\frac{\pt}{\pt s^i} + b^\alpha_i(z,s)\frac{\pt}{\pt z^\alpha}
+ c^\ol\beta_i(z,s)\frac{\pt}{\pt z^\ol\beta} \right)\right|_s
$$
be differentiable vector fields, whose projections to $S$ equal
$\frac{\pt}{\pt s^i}$. Then
$$
\frac{\pt}{\pt s^i} \int_{\cX_s} \eta =
\int_{\cX_s} L_{w_i}\left(\eta\right),
$$
where $L_{w_i}$ denotes the Lie derivative.
\end{lemma}

Concerning singular base spaces, it is obviously sufficient that the above
equation is given on the first infinitesimal neighborhood of $s$ in $S$.

\begin{proof}[Proof of {Lemma~\ref{le:intLie}}]
Because of linearity, one may consider the real and imaginary parts of
$\pt/\pt s^i$ and $w_i$ resp.\ separately.

Let $\pt/\pt t$ stand for $Re(\pt/\pt s^i)$, and let $\Phi_t: X \to \cX_t$
be the one parameter family of diffeomorphisms generated by $Re(w_i)$.
Then
$$
\frac{d}{d t} \int_{\cX_s} \eta = \int_X \frac{d}{dt} \Phi^*_t \eta = \int_X L_{Re(w^i)}(\eta).
$$
In general the vector fields $Re(w_i)$ and $Im(w_i)$ need not commute.
\end{proof}

\begin{remark}
{\rm Taking Lie derivatives of differential forms or tensors is in general
not type-preserving.}
\end{remark}

In our applications, the form $\eta$ will typically consist of inner
products of differential forms with values in hermitian vector bundles,
whose factors need to be treated separately. This will be achieved by Lie
derivatives. In this context, we will have to use covariant derivatives
with respect to the given hermitian vector bundle on the total space and
to the \ka metrics on the fibers. k

\subsection{Direct images and differential forms}\label{ss:dolb}
Let $(\cE, h)$ be a hermitian, holomorphic vector bundle on $\cX$, whose
direct image $R^q f_*\cE$ is {\em locally free}. Furthermore we assume
that for all $s\in S$ the cohomology $H^{q+1}(\cX_s, \cE \otimes
\cO_{\cX_s})$ {\em vanishes}. Then the statement of the
Grothendieck-Grauert comparison theorem holds for $R^q f_*\cE$, in
particular $R^q f_*\cE\otimes_{\cO_S} \C(s)$ can be identified with
$H^{q}(\cX_s, \cE \otimes_{\cO_\cX}\cO_{\cX_s})$. Since $\cE$ is $S$-flat,
by Grauert's theorem, for $R^q f_*\cE$ to be locally free, it is
sufficient to assume that $\dim_\C H^{q}(\cX_s, \cE
\otimes_{\cO_\cX}\cO_{\cX_s})$ is constant as long as the resp.\ base
space is reduced. Reducedness will always be assumed unless stated
differently.

For simplicity, we assume at this point that he base space $S$ is smooth.
Locally, after replacing $S$ by a neighborhood of any given point, we can
represent sections of the $q$-th direct image sheaf in terms of Dolbeault
cohomology by $\ol\pt$-closed $(0,q)$-forms with values in $\cE$. On the
other hand, fiberwise, we have harmonic representatives of cohomology
classes with respect to the \ka form and hermitian metric on the fibers.
The following fact will be essential. Under the above assumptions we have:

\begin{lemma}\label{le:dolb}
Let $\wt\psi \in R^q f_*\cE(S)$ be a section. Let $\psi_s \in
\cA^{0,q}(\cX_s,\cE_s)$ be the harmonic representatives of the cohomology
classes $\wt\psi|\cX_s$.

Then locally with respect to $S$ there exists a $\ol\pt$-closed form
$\psi\in \cA^{0,q}(\cX,\cE)$, which represents $\wt \psi$, and whose
restrictions to the fibers $\cX_s$ equal $\psi_s$.
\end{lemma}

\begin{proof}
For the sake of completeness we give the simple argument. The harmonic
representatives $\psi_s$ define a relative form, which we denote by
$\psi_{\cX/S}$. Let $\Phi \in \cA^{0,q}(\cX,\cE)$ represent $\wt \psi$.
Denote by $\Phi_{\cX/S}$ the induced relative form. Then there exists a
relative $(0,q-1)$-form $\chi_{\cX/S}$ on $\cX$, whose exterior derivative
in fiber direction $\ol\pt_{\cX/S}(\chi_{\cX/S})$ satisfies
$$
\psi_{\cX/S}= \Phi_{\cX/S} +\ol\pt_{\cX/S}(\chi_{\cX/S}).
$$
Let $\{U_i\}$ be a covering of $\cX$, which possesses a partition of unity
$\{\rho_i\}$ such that all $\chi_{\cX/S}|U_i$ can be extended to
$(0,q-1)$-forms $\chi_i$ on $U_i$. Then with $\chi=\sum\rho_i\chi_i$ we
set $\psi =\Phi + \ol\pt \chi$.
\end{proof}

The relative Serre duality can be treated in terms of such differential
forms. Let $\cE^\vee = \textit{Hom}_\cX(\cE,\cO_\cX)$. Then (under the
above assumptions)
$$
R^pf_*\cE \otimes_{\cO_S} R^{n-p}f_*(\cE^\vee \otimes_{\cO_\cX} \cK_{\cX/S}) \to \cO_S
$$
is given by the fiber integral of the wedge product of $\ol\pt$-closed
differential forms in the sense of Lemma~\ref{le:dolb}. By
\eqref{eq:fibintallg} (for the operator $\ol\pt$), the result is a
$\ol\pt$-closed $0$-form i.e.\ a holomorphic function.

\section{Estimates for resolvent and heat kernel}\label{se:heatker}
Let $(X,\omega_X)$ be a compact \ka manifold.
The Laplace operator for differentiable functions is given by $\Box =
\ol\pt\ol\pt^*+\ol\pt^*\ol\pt$, where the adjoint $\ol\pt^*$ is the formal
adjoint operator. The Laplacian is self-adjoint with non-negative
eigenvalues.

\begin{proposition}\label{pr:resol}
Let $(X,\omega_X)$ be a \ke manifold of constant Ricci curvature $-1$ with
volume element $g\, dV$. Then there exists a strictly positive function
$P_n(d(X))$, depending on the dimension $n$ of $X$ and the diameter $d(X)$
with the following property:

If $\chi$ is a non-negative continuous function and $\phi$ a solution of
\begin{equation}\label{eq:ellip}
(1+\Box)\phi = \chi,
\end{equation}
then
\vspace{-3.5mm}
\begin{equation}\label{eq:est}
\phi(z)\geq P_n(d(X))\cdot \int_X \chi \, g\/ dV
\end{equation}
for all $z\in X$.
\end{proposition}

The minimum principle immediately shows that $\phi\geq 0$. In
\cite{sch-preprint} we gave an elementary proof that $\chi\geq 0$, implies
that $\phi$ is strictly positive everywhere or identically zero for real
analytic functions $\chi$, which was sufficient for our application. Bo
Berndtsson pointed out to the author the existence of a refined minimum
principle yielding the same fact without the assumption of analyticity.

The explicit estimates claimed in the above Proposition are based on the
theory of resolvent kernels.

We summarize some known facts. Let $\{\psi_\nu\}$ be an orthonormal basis
of the space of square integrable real valued functions consisting of
eigenfunctions of the Laplacian with eigenvalues $\lambda_\nu$. Then the
resolvent kernel for the above equation is given by
$$
P(z,w)= \sum_\nu \frac{1}{1+\lambda_\nu} \psi_\nu(z)\psi_\nu(w),
$$
and the solution $\phi$ of \eqref{eq:ellip} equals
$$
\phi(z):=(\Box + id)^{-1}(\chi)(z)= \int_X P(z,w)\chi(w) g(w)dV_w.
$$
So the integral kernel $P(z,w)$ must be non-negative for all $z$ and $w$.

In a similar way we denote by $P(t,z,w)$ the integral kernel for the heat
operator
$$
\frac{d}{dt} + \Box
$$
so that the solution of the heat equation
$$
(\frac{d}{dt}+\Box )\phi = 0
$$
with initial function $\chi(z)$ for $t=0$ is given by
$$
\int_X P(t,z,w)\chi(w) g(w)dV_w.
$$
Here $ P(t,z,w)= \sum_\nu e^{-t\lambda_\nu} \psi_\nu(z)\psi_\nu(w). $ Now,
since the eigenvalues of the Laplacian are non-negative, the equation $
\int_0^\infty e^{-t(\lambda + 1)} dt = 1/(1+\lambda) $ implies the
following Lemma, (cf.\ also \cite[(3.13)]{c-y}).
\begin{lemma}
\label{le:heat} Let $P(z,w)$ be the integral kernel of the resolvent
operator and denote by $P(t,z,w)$ the heat kernel. Then
\begin{equation}\label{eq:resheatest}
P(z,w)= \int_0^\infty e^{-t}P(t,z,w)\, dt.
\end{equation}
\end{lemma}

We now apply the lower estimates for the heat kernel by Cheeger and Yau
\cite{c-y} to the resolvent kernel. Assuming constant negative Ricci
curvature $-1$, we use the estimates in the form of \cite[(4.3)
Corollary]{st}.
\begin{equation}\label{eq:hker}
P(t,z,w)\geq Q_{n}(t,r(z,w)):=\frac{1}{(2\pi t)^n} e^{- \frac{r^2(z,w)}{t}} e^{-\frac{2n-1}{4}t},
\end{equation}
Where $r=r(z,w)$ denotes the geodesic distance (and $n=\dim X$).

Let
\begin{equation}\label{eq:hker1}
P_{n}(r)= \int_0^\infty e^{-t} Q_n(t,r)\, dt>0.
\end{equation}

Using Lemma~\ref{le:heat}, equation \eqref{eq:hker}, and the monotonicity
of $P_n$ we get
\begin{equation}\label{eq:hker2}
P(z,w) \geq P_n(r(z,w)) \geq P_n(d(X)),
\end{equation}
where $d(X)$ denotes the diameter of $X$. This shows
Proposition~\ref{pr:resol}. \qed

We mention that $\lim_{r\to \infty}P_n(r)=0$.

Conversely, if we require that for all functions $\chi$ the solutions
$\phi$ of \eqref{eq:ellip} satisfy an estimate of the form
$$
\phi(z)\geq P \cdot \int_X \chi \, g\, dV,
$$
then $P\leq\inf P(z,w)$ follows immediately.

\medskip

\begin{tiny}
We mention that symbolic integration of \eqref{eq:hker} with
\eqref{eq:hker1} yields the explicit estimate
$$
P_{n}(r) \geq \frac{1}{(2\pi)^n} \frac{(2n+3)^{\frac{n-1}{2}}}{2^{n-2}} \frac{1}{r^{n-1}}
\text{BesselK}\left( n-1, \sqrt{2n+3} r   \right).
$$
\end{tiny}

\section{Positivity of $K_{\cX/S}$}\label{se:posi}
Let $X$ be a canonically polarized manifold of dimension $n$, equipped
with a \ke metric $\omega_X$. In terms of local holomorphic coordinates
$(z^1,\ldots, z^n)$ we write
$$
\omega_X=\ii g_{\alpha\ol\beta}(z)\; dz^\alpha\wedge dz^\ol\beta
$$
so that the \ke equation reads
\begin{equation}\label{eq:ke}
\omega_X=-{\rm Ric}(\omega_X),  \text{ i.e. }  \omega_X= \ddb \log g(z),
\end{equation}
where $g:=\det g_{\alpha\ol\beta}$. We consider $g$ as a hermitian metric
on the anti-canonical bundle $K_X^{-1}$.

For any holomorphic family of compact, canonically polarized manifolds $f:
\cX \to S$ of dimension $n$ with fibers $\cX_s$ for $s\in S$ the \ke forms
$\omega_{\cX_s}$ depend differentiably on the parameter $s$. The resulting
relative \ka form will be denoted by
$$
\omega_{\cX/S} = \ii g_{\alpha,\ol\beta}(z,s)\;dz^\alpha\wedge dz^\ol\beta.
$$
The corresponding hermitian metric on the relative anti-canonical bundle
is given by $g=\det \gab(z,s)$.  We consider the real $(1,1)$-form
$$
\omega_\cX= \ddb \log g(z,s)
$$
on the total space $\cX$. We will discuss the question, whether
$\omega_\cX$ is a \ka form on the total space.

The \ke equation \eqref{eq:ke} implies that
$$
\omega_\cX|\cX_s = \omega_{\cX_s}
$$
for all $s\in S$. In particular $\omega_\cX$, restricted to any fiber, is
positive definite. Our  result is the following statement (cf.\ Main
Theorem).

\begin{theorem}\label{th:main}
Let $\cX \to S$ be a holomorphic family of canonically polarized, compact,
complex manifolds. Then the curvature form $\omega_\cX$ of the hermitian
metric on $\cK_{\cX/S}$ induced by the \ke metrics on the fibers is
semi-positive and strictly positive on all fibers. It is strictly positive
in horizontal directions, for which the family is not infinitesimally
trivial.

More precisely, in terms of the generalized \wp form $\omega^{WP}$:
\begin{equation}
\omega^{n+1}_\cX \geq P_n(d(\cX_s)) \cdot f^*\omega^{WP}\wedge \omega^n_{\cX},
\end{equation}
where the right-hand side only depends on the relative volume form
$\omega^n_{\cX/S}$.
\end{theorem}
(The form $\omega^{WP}$ will be given in Definition~\ref{de:wpherm}. We
denote by $\omega_\cX^n$ etc.\ the $n$-fold exterior power divided by
$n!$, cf.\ also the remark on notation in Section~\ref{se:compcurv}.)

Both the statement of the theorem and the methods are valid for smooth,
proper families over singular (even non-reduced) complex \break spaces
(for the necessary theory cf.\ \cite[\S 12]{f-s:extremal}).

It is sufficient to prove the theorem for base spaces of dimension one
assuming $S\subset \C$. (In order to treat singular base spaces, the claim
can be reduced to the case where the base is a double point $(0,\mathbb
C[s]/(s^2))$. The arguments below will still be meaningful and can be
applied literally.)

We denote the \ks map for the family $f:\cX \to S$ at a given point
$s_0\in S$ by
$$
\rho_{s_0} :T_{s_0}S \to H^1(X, \cT_X)
$$
where $X=\cX_{s_0}$. The family is called {\it effectively parameterized}
at $s_0$, if $\rho_{s_0}$ is injective. The \ks map is induced by the edge
homomorphism for the short exact sequence
$$
0 \to  \cT_{\cX/S} \to \cT_\cX \to f^*\cT_S \to 0.
$$
If $v \in T_{s_0}S$ is a tangent vector, say $v=\frac{\pt}{\pt s}|_{s_0}$,
and if $\frac{\pt}{\pt s} + b^\alpha \frac{\pt}{\pt z^\alpha}$ is any lift
of class $\cinf$ to $\cX$ along $X$, then
$$
\ol\pt\left(\frac{\pt}{\pt s} + b^\alpha(z) \frac{\pt}{\pt z^\alpha}\right)=
\frac{\partial b^\alpha(z)}{\partial z^\ol\beta}
\frac{\pt}{\pt z^\alpha} dz^\ol\beta
$$
is a $\ol\pt$-closed form on $X$, which represents $\rho_{s_0}(\pt / \pt
s)$.

We will use the semi-colon notation as well as raising and lowering of
indices for covariant derivatives with respect to the {\it \ke metrics on
the fibers}. The $s$-direction will be indicated by the index $s$. In this
sense the coefficients of $\omega_\cX$ will be denoted by $g_{s\ol s}$,
$g_{\alpha\ol s}$, $\gab$ etc.

Next, we define {\it canonical lifts} of tangent vectors of $S$ as
differentiable vector fields on $\cX$ along the fibers of $f$ in the sense
of Siu \cite{siu:canlift}. By definition these satisfy  the property that
the induced representatives of the \ks class are {\it harmonic}.

Since the form $\omega_\cX$ is positive, when restricted to fibers, {\em
horizontal lifts} of tangent vectors with respect to the pointwise
sesquilinear form $\langle-,-\rangle_{\omega_\cX}$ are well-defined  (cf.\
also \cite{sch:curv}).
\begin{lemma}\label{le:canlift}
The horizontal lift of $\pt/\pt s$  equals
$$
v = \pt_s + a_s^\alpha \pt_\alpha,
$$
where
$$
a_s^\alpha = - g^{\ol\beta \alpha} g_{s \ol \beta}.
$$
\end{lemma}

\begin{proposition}\label{pr:harmrep}
The horizontal lift  of a tangent vector of the base induces the harmonic
representative of its \ks class.
\end{proposition}

\begin{proof}
The \ks form of the tangent vector $\pt/\pt_{s}$ is $\ol\pt v|\cX_s =
a^\alpha_{s;\ol\beta}\pt_\alpha dz^\ol\beta$. We consider the tensor
$$
A^\alpha_{s\ol\beta}:= a^\alpha_{s;\ol\beta}|{\cX_{s_0}}
$$
on $X$. Then $\ol\pt^*(A^\alpha_{s\ol\beta}\pt_\alpha dz^{\ol\beta})$ is
given by
\begin{gather*}
g^{\ol\beta\gamma} A^\alpha_{s\ol\beta;\gamma}= - g^{\ol\beta\gamma}
g^{\ol\delta \alpha} g_{s\ol\delta;\ol\beta\gamma} = - g^{\ol\beta\gamma}
g^{\ol\delta \alpha} g_{s\ol\beta;\ol\delta\gamma} =
-g^{\ol\beta\gamma}g^{\ol\delta \alpha} \left(
g_{s\ol\beta;\gamma\ol\delta} -
g_{s\ol\tau}R^\ol\tau_{\; \ol\beta\ol\delta\gamma} \right)\\
=-g^{\ol\delta\alpha}\left(\left({\pt\log g}/{\pt s} \right)_{;\ol\delta}
+ g_{s\ol \tau}R^\ol\tau_{\; \ol\delta}\right) = 0
\end{gather*}
because of the \ke property.
\end{proof}
It follows immediately from the proposition that the harmonic \ks forms
induce symmetric tensors. This fact reflects the close relationship
between the \ks tensors and the \ke metrics:
\begin{corollary}\label{co:symm}
Let $A_{s \ol\beta\,\ol\delta}= g_{\alpha\ol\beta}A^\alpha_{s\ol\delta}$.
Then
\begin{equation}
A_{s \ol\beta\,\ol\delta}=A_{s \ol\delta\,\ol\beta}.
\end{equation}
\end{corollary}

Next, we introduce a {\it global} function $\varphi(z,s)$, which is by
definition the {\em pointwise inner product of the canonical lift $v$ of
$\pt/\pt s$ at $s\in S$} with itself taken with respect to $\omega_\cX$.
\begin{definition}\label{de:varphi}
\begin{equation}\label{eq:varphidef}
\varphi(z,s) := \langle \pt_s + a_s^\alpha \pt_\alpha, \pt_s + a_s^\beta
\pt_\beta  \rangle_{\omega_\cX}
\end{equation}
\end{definition}

Since $\omega_\cX$ is not known to be positive definite in all directions,
$\varphi$ is not known to be non-negative at this point.
\begin{lemma}
\begin{equation}\label{eq:varphi}
\varphi =  g_{s\ol s} - g_{\alpha\ol s}
g_{s\ol\beta} g{^{\ol\beta\alpha}}
\end{equation}
\end{lemma}
\begin{proof}
The proof follows from Lemma~\ref{le:canlift} and
$$
\varphi = g_{s\ol s} + g_{s\ol\beta}a^\ol\beta_{\ol s} + a_s^\alpha g_{\alpha\ol s}
+ a_s^\alpha a_{\ol s}^\ol\beta \gab.
$$
\end{proof}

Denote by $\omega^{n+1}_\cX$ the $(n+1)$-fold exterior product, divided by
$(n+1)!$ and by $dV$ the Euclidean volume element in fiber direction. Then
the global real function $\varphi$ satisfies the following property:
\begin{lemma}\label{le:varphi}
$$
\omega^{n+1}_\cX= \varphi \cdot g  dV\ii ds\wedge \ol{ds}.
$$
\end{lemma}

\begin{proof}
Compute the following $(n+1)\times(n+1)$-determinant
$$ \det
\left(
\begin{array}{cc}
g_{s\ol s} & g_{s\ol\beta}\\ g_{\alpha\ol s}& \gab
\end{array}
\right),
$$
where $\alpha,\beta=1,\ldots,n$.
\end{proof}

So far we are looking at {\it local} computations, which essentially only
involve derivatives of certain tensors. The only {\it global ingredient}\/
is the fact that we are given global solutions of the \ke equation.

The essential quantity is the differentiable function $\varphi$ on $\cX$.
Restricted to any fiber it associates the yet to be proven positivity of
the hermitian metric on the relative canonical bundle with the canonical
lift of tangent vectors, which is related to the harmonic \ks forms.

We use the Laplacian operators $\Box_{g,s}$ with non-negative eigenvalues
on the fibers $\cX_s$ so that for a real valued function $\chi$ the
Laplacian equals $\Box_{g,s}\chi
 = - g^{\ol\beta\alpha}\chi_{;\alpha\ol\beta}$.
\begin{proposition}\label{pr:elleq}
The following elliptic equation holds fiberwise:
\begin{equation}\label{eq:phiA}
(\Box_{g,s} + {\rm id})\varphi(z,s) = \|A_s(z,s)\|^2,
\end{equation}
where
$$
A_s=A^\alpha_{s\ol\beta} \frac{\pt}{\pt z^\alpha}dz^\ol\beta.
$$
\end{proposition}
\begin{proof}
In order to prove such an elliptic equation for $\varphi$ on the fibers,
we need to eliminate the second order derivatives with respect to the base
parameter. This is achieved by the left hand side of \eqref{eq:phiA}.
First,
\begin{eqnarray*}
g^{\ol\delta\gamma}g_{s\ol s;\gamma\ol\delta}
&=&g^{\ol\delta\gamma}\partial_s\partial_\ol s g_{\gamma\ol\delta}\\
&=&\partial_s(g^{\ol\delta\gamma}\partial_\ol s g_{\gamma\ol\delta})
 -a_s^{\gamma;\ol\delta}\partial_\ol s g_{\gamma\ol\delta}\\
&=&\partial_s\partial_\ol s \log g
 +a_s^{\gamma;\ol\delta} a_{\ol s\gamma;\ol\delta}\\
&=& g_{s \ol s}
 +a_s^\sigma{}_{;\gamma} a_{\ol s\sigma;\ol\delta} g^{\ol\delta\gamma}.
\end{eqnarray*}
Next
\begin{eqnarray*}
(a_s^\sigma a_{\ol s\sigma})_{;\gamma\ol\delta}g^{\ol\delta\gamma}
&=\left(a_s^\sigma{}_{;\gamma\ol\delta} a_{\ol s\sigma}
        +A_{s\ol\delta}^\sigma A_{\ol s\sigma\gamma}
        +a_{s;\gamma}^\sigma a_{\ol s\sigma;\ol\delta}
        +a_s^\sigma A_{\ol s\sigma\gamma;\ol\delta}
\right)g^{\ol\delta\gamma} .
\end{eqnarray*}
The last term vanishes because of the harmonicity of $A_s$, and
\begin{eqnarray*}
a_{s;\gamma\ol\delta}^\sigma g^{\ol\delta\gamma}
&=&A_{s\ol\delta;\gamma}^\sigma g^{\ol\delta\gamma}
  +a_s^\lambda R^\sigma{}_{\lambda\gamma\ol\delta}g^{\ol\delta\gamma}\\
&=&0-a_s^\lambda R^\sigma{}_\lambda\\
 &=& a_s^\sigma .
\end{eqnarray*}
\end{proof}
\begin{proof}[Proof of Theorem~\ref{th:main}]
It is sufficient to show that the function $\varphi$ from
Definition~\ref{de:varphi} is strictly positive, since
$\omega_\cX|\cX_s=\omega_{\cX_s}$ is known to be positive definite. The
positivity of $\varphi$ follows from \eqref{eq:phiA} in
Proposition~\ref{pr:elleq} together with Proposition~\ref{pr:resol}.
\end{proof}

\begin{definition}\label{de:wpherm}
The \wp hermitian product on $T_sS$ is given by the $L^2$-inner product of
harmonic \ks forms:
\begin{equation}\label{eq:wpherm}
\Big\|\frac{\pt}{\pt s}\Big\|^2_{WP}:= \int_{\cX_s} A^\alpha_{s\ol\beta} A^\ol\delta_{\ol s\gamma }
g_{\alpha\ol\delta}g^{\ol\beta\gamma} g \, dV =
\int_{\cX_s} A^\alpha_{s\ol\beta} A^\ol\beta_{\ol s\alpha } g \, dV
\end{equation}
If the tangent vectors ${\pt}/{\pt s^i}\in T_sS$ are part of a basis, we
denote by $G^{WP}_{i\ol\jmath}(s)$ the inner product $\langle
\pt/\pt_i|_s, \pt/\pt_j|_s \rangle_{WP}$, and set
$$
\omega^{WP}:= \ii G^{WP}_{i\ol\jmath} ds^i\we ds^{\ol\jmath}.
$$
\end{definition}
Observe that the generalized \wp form is equal to a fiber integral (cf.\
{\cite[Theorem 7.8]{f-s:extremal}}). The above approach yields a simple
proof of this fact, which also implies the \ka property of $\omega^{WP}$:
\begin{proposition}\label{pr:fibint}
\begin{equation}\label{eq:wpfib}
\omega^{WP} = \int_{\cX/S} \omega^{n+1}_\cX.
\end{equation}
\end{proposition}
The {\it proof}\/ follows from Lemma~\ref{le:varphi} and
Proposition~\ref{pr:elleq}.

\section{Fiber integrals and Quillen metrics}\label{se:bgs}
In this section we summarize the methods how to produce a positive line
bundle on the base of a holomorphic family from \cite[\S 10]{f-s:extremal}.
Let $f:\cX \to S$ be a proper, smooth holomorphic map of reduced complex
spaces and $\omega_{\cX/S}$ a closed real $(1,1)$-form on $\cX$, whose
restrictions to the fibers are \ka forms. Let $(\cE, h)$ be a hermitian
vector bundle on $\cX$. We denote the determinant line bundle of $\cE$ in
the derived category by
$$
\lambda(\cE)=\det f_!(\cE).
$$
The main result of Bismut, Gillet and Soulé from \cite{bgs} states the
existence of a Quillen metric $h^Q$ on the determinant line bundle such
that the following equality holds for its Chern form on the base $S$ and
the component in degree two of a fiber integral:
\begin{equation}\label{eq:bgs}
c_1(\lambda(\cE),h^Q)= \left[\int_{\cX/S}\textit{td}
(\cX/S,\omega_{\cX/S})\textit{ch}(\cE,h)\right]_2
\end{equation}
Here $\textit{ch}$ and $\textit{td}$ stand for the Chern and Todd
character forms.

The formula is applied to a virtual bundle of degree zero (cf.\
\cite[Remark 10.1]{f-s:extremal} for the notion of virtual bundles). It
follows immediately from the definition that the determinant line bundle
of a virtual vector bundle is well-defined (as a line bundle).

Let $(\cL,h)$ be a hermitian line bundle, and set $\cE=(\cL-
\cL^{-1})^{n+1}$. The difference is taken in the Grothendieck group, and
the product is the tensor product. Since the term of degree zero in the
Chern character $\text{ch}(\cL- \cL^{-1})$ is equal to the (virtual) rank,
which is zero, and the first term is $2c_1(\cL)$, we conclude that
$\text{ch}(\cE)$ is equal to
\begin{equation}\label{eq:fib}
2^{n+1} c_1(\cL)^{n+1}
\end{equation}
plus higher degree terms. Furthermore, the hermitian structure on $\cL$
provides $\cE$ with a natural Chern character form.

Hence the only contribution of the Todd character form in \eqref{eq:bgs}
is the constant $1$ resulting in the following equality
\begin{equation}\label{eq:fib0}
c_1(\lambda(\cE),h^Q) = 2^{n+1} \int_{\cX/S}c_1(\cL,h)^{n+1}.
\end{equation}

Following \cite[Theorem 11.10]{f-s:extremal} we apply this construction to
families of canonically polarized varieties. Let $f:\cX \to S$ be any
(smooth) family of canonically polarized manifolds over a reduced complex
space. The generalized \wp form $\omega^{WP}_S$ on $S$ was proven to be
equal to a certain fiber integral. We will use the notion $\simeq$ for
equality up to a numerical factor.

We set $\cL=\cK_{\cX/S}$ in \eqref{eq:bgs}. Equation \eqref{eq:fib0}
yields
\begin{equation}\label{eq:fibint}
c_1(\lambda(\cE),h^Q) \simeq \int_{\cX/S} \omega_\cX^{n+1},
\end{equation}
where $\omega_\cX= c_1(\cK_{\cX/S},h)$, with $h$ induced by the \ke volume
forms on the fibers.

On the other hand Lemma~\ref{le:varphi} together with
Proposition~\ref{pr:elleq} implied that the fiber integral
\eqref{eq:fibint} is equal to the generalized \wp form:
$$
\omega^{WP}(s) = \int_{\cZ_s} A^\alpha_{i\ol\beta} A^\ol\delta_{\ol\jmath\gamma}
g_{\alpha\ol\delta}g^{\ol\beta\gamma} g \, dV \ii ds^i\wedge ds^\ol\jmath,
$$
where the forms $A^\alpha_{i\ol\beta}\pt_\alpha dz^{\ol\beta} $ are the
harmonic representatives of the \ks classes $\rho_s({\pt/\pt s_i}|_s)$,
i.e.\
\begin{equation}\label{eq:fibint_ke}
\omega^{WP}_S\simeq \int_{\cX/S} c_1({\cK_{\cX/S}}, h)^{n+1}.
\end{equation}
Now
\begin{equation}\label{eq:detbdl} c_1(\det f_!
((\cK_{\cX/S}-\cK_{\cX/S}^{-1})^{n+1}), h^Q) \simeq \omega^{WP}_S.
\footnote{Also the element $(\cK_{\cX/S}-\cO_\cX)^{n+1}$ of the
relative Grothendieck group can be taken instead.}
\end{equation}

We consider the situation of Hilbert schemes of canonically polarized
varieties.

After fixing the Hilbert polynomial and a multiple $m$ of the canonical
bundles in the family that yields very ampleness, we consider the
universal embedded family over the Hilbert scheme
$$
\xymatrix{\cX \ar[r]^i \ar[dr]_f & \mathbb P_N \times \cH \ar[d]^{pr} \\ & \cH.
}
$$
In this sense we modify the determinant line bundle and consider
$$
\lambda=\det f_! ((\cK_{\cX/\cH}^{\otimes m}-(\cK_{\cX/\cH}^{-1})^{\otimes m})^{n+1}),
$$
which only yields an extra factor $m^{n+1}$ in front of the \wp form
$\omega^{WP}$ on $\cH$.

We point out, how singularities of base spaces (and smooth maps) were
treated in \cite[Theorem 10.1 and \S 12]{f-s:extremal}: Given a (local)
deformation of a canonically polarized variety, equipped with a \ke metric
over a reduced singular base, the latter is embedded into a smooth ambient
space. The deformation is computed in terms of certain elliptic operators,
which are meaningful for all neighboring points. The integrability
condition for the respective almost complex structures, i.e.\ the
vanishing of the Nijenhuis torsion tensor, determines the singular base
space. Now the implicit function theorem yields \ke metrics on neighboring
fibers -- the solutions also exist (without being too significant) for
points of the base, where the almost complex structure is not integrable.
This procedure yields a potential for the relative \ke forms that comes
from a differentiable function on the smooth ambient space. By fiber
integration the \wp form is being computed; again, it possesses a
$\pt\ol\pt$-potential, which is the restriction of a $\cinf$-function on
the smooth ambient space.

\section{An extension theorem for hermitian line
bundles\\and positive currents}\label{se:extlinebundles} Let $\cL$ be a
holomorphic line bundle on a reduced, complex space $\cZ$. Then a
semi-positive, singular hermitian metric $h$ on $\cL$ is defined by the
property that the locally defined function $-\log h$ is plurisubharmonic
(and locally integrable), when pulled back to the normalization of the
space. By definition, a {\em positive current} takes non-negative values
on semi-positive differential forms.

\begin{definition} Let $\chi$ be a plurisubharmonic function on an open subset of
$\C^n$. We say that $\chi$ has {\em at most analytic singularities}, if
locally
$$
\chi \geq \gamma\log \sum^{k}_{\nu=1} |f_\nu|^2 + const.
$$
holds, for holomorphic functions $f_\nu$ and some $\gamma>0$. In this
situation, we say that a (locally defined) positive, singular hermitian
metric of the form
$$
h=e^{-\chi}
$$
has at most analytic singularities. This property will be also assigned to
a locally $\pt\ol\pt$-exact, positive current of the form
$$
\omega=\ii \pt\ol\pt \chi.
$$
\end{definition}

For any positive closed $(1,1)$-current $T$ on a complex manifold $Y$ the
Lelong number at a point $x$ is denoted by $\nu(T,x)$, and for any $c>0$
we have the associated sets $E_c(T)=\{ x; \nu(T,x)\geq c\}$. According to
\cite[Main Theorem]{siu:curr} these are closed analytic sets.

We will use in an essential way Siu's decomposition formula for  positive,
closed currents on complex manifolds. We state it for $(1,1)$-currents.
\begin{theorem*}[{\cite{siu:curr}}]
Let $\omega$ be a closed positive $(1,1)$-current. Then $\omega$ can be
written as a series of closed positive currents
\begin{equation}\label{eq:decomp}
\omega = \sum_{k=0}^\infty \mu_k [Z_k] + R,
\end{equation}
where the $[Z_k]$ are currents of integration over irreducible analytic
sets of codimension one, and R is a closed positive current with the
property that $\dim E_c(R) < \dim Y-1$ for every $c > 0$. This
decomposition is locally and globally unique: the sets $Z_k$ are precisely
components of codimension one occurring in the sublevel sets
$E_c(\omega)$, and $\mu_k = \min_{x\in Z_k} \nu(\omega; x)$ is the generic
Lelong number of $\omega$ along $Z_k$.
\end{theorem*}

We want to prove an extension theorem for hermitian line bundles, whose
curvature forms $\omega$ extend as positive currents. The idea is to treat
the non-integer part
$$
\sum_{k=0}^\infty (\mu_k - \lfloor \mu_k\rfloor)[Z_k].
$$
of the decompositions in \eqref{eq:decomp}.

\begin{theorem}\label{th:extlinebdl}
Let $Y$ be a normal complex space and $Y' = Y \setminus A$ the complement
of a nowhere dense, closed, analytic subset. Let $L'$ be a holomorphic
line bundle together with a hermitian metric $h'$ of semi-positive
curvature, which also may be singular. Assume that the curvature current
$\omega'$ of $(L',h')$ possesses an extension $\omega$ to $Y$ as a closed,
positive current. Then there exists a holomorphic line bundle $(L,h)$ with
a singular, positive hermitian metric, whose restriction to $Y'$ is
isomorphic to $(L',h')$. If $\omega$ has at most analytic singularities,
then $h$ can be chosen with this property.
\end{theorem}

We first note the following special case:

\begin{proposition}\label{pr:extlin}
The theorem holds, if $Y$ is a complex manifold and $A$ is a simple normal
crossings divisor.
\end{proposition}

\begin{proof}
We will argue in an elementary way. We first assume that $A\subset Y$ is a
smooth, connected hypersurface. Let $\{U_j \}$ be an open covering of $Y$
such that the set $A\cap U_i$ consists of the zeroes of a holomorphic
function $z_i$ on $U_i$. Since all holomorphic line bundles on the product
of a polydisk and a punctured disk are trivial, we can chose the sets
$\{U_i\}$ such that the line bundle $L'$ extends to such $U_j$ as a
holomorphic line bundle. So $L'$ possesses nowhere vanishing sections over
$U'_j=U_j\backslash A$. Hence $L'$ is given by a cocycle $g_{ij}' \in
\cO_Y^*(U_{ij}')$, where $U_{ij}=U_i\cap U_j$ and $U_{ij}'=U_{ij}\cap Y'$.
If necessary, we will replace $\{U_i\}$ by a finer covering.

We will first show the existence of \psh\ functions $\psi_i$ on $U_i$ and
holomorphic functions $\varphi'_i$ on $U'_i$ such that
$$
h'_i\cdot|\varphi'_i|^2 = e^{-\psi_i}|U'_i
$$
where the $\psi_i-\psi_j$ are pluriharmonic functions on $U_{ij}$:

Let
$$
\omega|U_i= \ddb (\psi^0_i)
$$
for some \psh\ functions $\psi^0_i$ on $U_i$. Now
$$
\log(e^{\psi^0_i} h'_i)
$$
is pluriharmonic on $U'_i$. For a suitable number $\beta_i\in \R$ and some
holomorphic function $f'_i$ on $U'_i$ we have
\begin{equation}\label{eq:period}
\log(e^{\psi^0_i} h'_i) + \beta_i \log|z_i| = f'_i+ \ol{f'_i}.
\end{equation}
We write
$$
\beta_i=\gamma_i + 2k_i
$$
for $0\leq \gamma_i <2$ and some integer $k_i$. We set
\begin{equation}\label{eq:psi}
\psi_i = \psi^0_i + \gamma_i \log |z_i|.
\end{equation}
These functions are clearly \psh, and $\gamma_i \log |z_i|$ contributes as
an analytic singularity to $\psi^0_i$. Set
$$
\varphi'_i= z^{-k_i}_i e^{f'_i} \in \cO^*(U'_i).
$$
We use the functions $\varphi'_i$ to change the bundle coordinates of $L'$
with respect to $U'_i$. In these bundle coordinates the hermitian metric
$h'$ on $L'$ is given by
$$
\wt h'_i = h'_i \cdot |\varphi'_i|^{-2}
$$
and the transformed transition functions are
$$
\wt g'_{ij}= {\varphi'_i}\cdot g'_{ij}\cdot ({\varphi'_j})^{-1}.
$$
Now
\begin{equation}\label{eq:wth}
\wt h'_i = |z_i|^{-\gamma_i} e^{-\psi^0_i}|U'_i= e^{-\psi_i}|U'_i ,
\end{equation}
and
\begin{equation}\label{eq:wtg}
|\wt g'_{ij}|^2 = \wt h'_j (\wt h'_i)^{-1} =
|z_j|^{\gamma_i-\gamma_j} \left|\frac{z_i}{z_j}\right|^{\gamma_i} \cdot
e^{\psi^0_i-\psi^0_j}.
\end{equation}
Since the function $z_i/z_j$ is holomorphic and nowhere vanishing  on
$U_{ij}$, and since the function $\psi^0_i-\psi^0_j$ is pluriharmonic on
$U_{ij}$, the function
$$
\left|\frac{z_i}{z_j}\right|^{\gamma_i} \cdot
e^{\psi^0_i-\psi^0_j}
$$
is of class $\cinf$ on $U_{ij}$ with no zeroes. Now
$-2<\gamma_i-\gamma_j<2$, and $\wt g'_{ij}$ is holomorphic on $U'_{ij}$.
So we have
\begin{equation}\label{eq:beta}
\gamma_i=\gamma_j=:\gamma_A
\end{equation}
(whenever $A\cap U_{ij} \neq \emptyset)$. Accordingly \eqref{eq:wtg} reads
\begin{equation}\label{eq:wtg1}
|\wt g'_{ij}|^2 = \wt h'_j (\wt h'_i)^{-1} =
\left|\frac{z_i}{z_j}\right|^{\gamma_i} \cdot
e^{\psi^0_i-\psi^0_j},
\end{equation}
and the transition functions $\wt g'_{ij}$ can be extended holomorphically
to all of $U_{ij}$. So a line bundle $L$ exists.

The functions $\psi_i$ are \psh, and the quantity
$$
\wh \omega = \ddb \psi_i = \omega + {\pi}\,{\gamma_A}\,[A]
$$
is a well-defined positive current on $Y$ because of \eqref{eq:psi} and
\eqref{eq:beta}. Its restriction to $Y'$ equals $\omega|Y'$.

We define
$$
\wt h_i= e^{-\psi_i} = e^{-\psi^0_i}|z_i|^{-\gamma_i}
$$
on $U_i$. It defines a positive, singular, hermitian metric on $L$. This
shows the theorem in the special case.

Observe that the numbers $\beta_j=\gamma_j+ 2k_j$ only depend on the
smooth hypersurface $A$. We will denote these by $\beta_A$ etc.

For the general case of the proposition, we assume for simplicity first
that the divisor $A$ consists of two components $A_1$ and $A_2$, which
intersect transversally. Let $Y^j=Y\backslash A_j$ for $j=1,2$. We know
from the first part that there exist two line bundle extensions
$(Y^1,h^1)$ and $(Y^2,h^2)$ resp.\ of $(L',h')$ to $Y^1$ and $Y^2$ resp.

\medskip
{\em Claim.} The hermitian line bundles $(L^j,h^j)$ define the extension
of $(L',h')$ as a holomorphic line bundle on $Y$ with a singular hermitian
metric with positive curvature equal to
$$
\omega +  \pi \,\gamma_{A_1}\, [A_1] + \pi\, \gamma_{A_2}\, [A_2],
$$
where $0\leq\gamma_{A_j}<2$ are given by the first part of the proof.

\medskip

We prove the {\em Claim}. Let $p\in A_1\cap A_2$. For simplicity we may
assume that $\dim Y=2$. Let $U(p)= \Delta\times\Delta =\{(z_1,z_2)\}$ be a
neighborhood of $p$, where $\Delta\subset \C$ denotes the unit disk.
Assume that $A_j=V(z_j)$, $j=1,2$. Now the metric $h^1$ is defined on
$\Delta^*\times \Delta$ and $h^2$ is defined on $\Delta\times \Delta^*$,
where $\Delta^*=\Delta\backslash \{0\}$. Since any holomorphic line bundle
on $\Delta^*\times \Delta$ is trivial, we may use the spaces
$\Delta^*\times \Delta$ and $\Delta\times \Delta^*$ as coordinate
neighborhoods for the definition of the line bundles $L^1$ and $L^2$ resp.
Since holomorphic line bundles on $\Delta\times \Delta^*$ and
$\Delta^*\times \Delta$ extend to $\Delta\times\Delta$, we find a nowhere
vanishing function  $\kappa\in \cO^*_{\C^2}(\Delta^*\times \Delta^*)$ such
that
\begin{equation}\label{eq:h1h2}
h_1= |\kappa|^2h_2.
\end{equation}
With the above methods it is easy to show that any such function satisfies
\begin{equation}\label{eq:kappa}
\kappa(z_1,z_2)=z^{m_1}_1 z^{m_2}_2 e^\chi
\end{equation}
for a holomorphic function $\chi\in \cO_{\C^2}(\Delta^*\times\Delta^*)$
and $m_j\in \mathbb Z$. (In order to see this, write $\log|\kappa|^2 =
\sigma_1\log|z_1| + \sigma_2\log|z_2| + \phi + \ol\phi$, with
$\phi\in\cO_{\C^2}(\Delta^*\times\Delta^*)$. Then $|\kappa \, e^{-\phi} \,
z_1^{-q_1}\, z_2^{-q_2}|^2 = |z_1|^{\tau_1}\, |z_2|^{\tau_2}$ with
$\sigma_j=\tau_j + 2 q_j$, $0\leq \tau_j <2$, $q_j \in \mathbb Z$. Now
$\kappa \, e^{-\phi}\, z_1^{-q_1}\, z_2^{-q_2}$ possesses a holomorphic
extension to $\Delta\times \Delta$, and $\tau_1=\tau_2=0$.)

We use the arguments and result of the first part, and again the fact that
the homotopy group $\pi_1(\Delta^n \backslash V(z_1\cdot \ldots \cdot
z_k))$ for any $k\leq n$ is abelian. We find
\begin{eqnarray}
-\log h^1 &=& \psi^0 + \beta_1\log|z_1| + \gamma_{A_2}\log|z_2| +f_1+\ol f_1
\label{eq:h1}\\
-\log h^2 &=& \psi^0 + \gamma_{A_1}\log|z_1| + \beta_2\log|z_2| +   f_2+\ol f_2,
\label{eq:h2}
\end{eqnarray}
where $\beta_j\in \R$, and $f_j \in \cO_{\C^2}(\Delta^*\times\Delta^*)$.
The numbers $0\leq \gamma_{A_j}<2$ are already determined. Let
$$
\beta_j=\gamma_j + 2\ell_j \text{ with } 0\leq \gamma_j <2 \text{ and } \ell_j \in \mathbb Z.
$$
Now \eqref{eq:h1h2}, \eqref{eq:kappa}, \eqref{eq:h1}, and \eqref{eq:h2}
imply that
$$
\gamma_j=\gamma_{A_j}, \, \ell_1=-m_1, \, \ell_2=m_2,
$$
and
$$
e^{f_2-f_1-\chi}
$$
possesses a holomorphic extension to $\Delta\times \Delta$. Now, like in
the first part, the functions $z^{\ell_1-m_1}_1$ and $z^{\ell_2+m_2}_2$,
together with the function $e^{f_2-f_1-\chi}$ can be used as coordinate
transformations for the line bundles $L_1$ and $L_2$ on $\Delta^*\times
\Delta$ and $\Delta\times \Delta^*$ resp. This shows the claim.

The case of a general, simple normal crossings divisor follows in an
analogous way.
\end{proof}

The proof of Proposition~\ref{pr:extlin} implies the following fact
(introduce extra auxiliary local smooth divisors).

\begin{proposition}\label{pr:linext}
Let $n=\dim Y$.  Then the statement of the above proposition still holds,
when $A\subset Y$ is an analytic set with smooth irreducible components
and transverse intersections such that no more than $n$ components meet at
any point. The curvature current of the extended singular hermitian metric
differs from the given current only by an added sum
\begin{equation}\label{eq: sumgamA}
{\pi}\sum_j \gamma_j [A_j], \text{ where }0\leq \gamma_j<2 \,.
\end{equation}
If the given line bundle already possesses an extension into a component
of $A$, together with an extension of the singular hermitian metric such
that $\omega$ is the curvature current, then the above construction
reproduces these without any additional currents of integration.
\end{proposition}

\begin{proof}[Proof of the Theorem.]
We first mention that because of Proposition~\ref{pr:extlin} and
Proposition~\ref{pr:linext} we may assume that $A\subset Y$ is of
codimension at least two: Namely we extend both the line bundle and the
singular hermitian metric into the locus $Y''$, where $A$ is of
codimension one and $Y$ is smooth. The original current $\omega$ is being
changed in this way by adding currents of integration of the form
\eqref{eq: sumgamA}, where the $A_j$ are components of $A$ of codimension
one. These currents can obviously be extended from $Y''$ to $Y$.

The remaining case should also be seen relating to the extension theorem
of Shiffman \cite{shi} for positive line bundles on normal spaces.

We first take a desingularization $\tau:\wt Y \to Y$ and consider $\wt
A=\tau^{-1}(A)$. Next we choose a modification $\mu : Z \to \wt Y$ that
defines an embedded resolution of singularities of $\wt A \subset \wt Y$.
In particular $\mu^{-1}(\wt A)= B \cup E$ is a transversal union, where
the proper transform $B$ of $\wt A$ is the desingularization and the
normal crossings divisor $E$ is the exceptional locus of $\mu$. (The
divisorial component of $B$ together with $E$ is a simple normal crossings
divisor).

We pull back the line bundle $L'$ to $Z \backslash \mu^{-1}(\wt A)$
together with the given data. We apply Proposition~\ref{pr:extlin} and
Proposition~\ref{pr:linext}, and obtain an extension $\wt\cL$ of the line
bundle and of the singular hermitian metric  $h_{\wt\cL}$ of positive
curvature. At those places, where an extended line bundle with a singular
hermitian metric already exists, the construction yields the pull-back.
The determinant line bundle $\det((\tau\circ \mu)_!\wt\cL)$ defines an
extension of $L'$. Observe that the original line bundle over $Y'$ is
reproduced together with the singular hermitian metric. Because of the
normality of $Y$ an extension into an analytic set of codimension two or
more is unique, if it exists, and the singular metric of positive
curvature can be extended.
\end{proof}

If $Y$ is just a reduced complex space, we still have the following
statement.

\begin{proposition}\label{pr:nonnormal}
Let $Y$ be a reduced complex space, and $A\subset Y$ a closed analytic
subset. Let $\cL$ be an invertible sheaf on $Y \backslash A$, which
possesses a holomorphic extension to the normalization of\/ $Y$ as an
invertible sheaf. Then there exists a reduced complex space $Z$ together
with a finite map $Z \to Y$, which is an isomorphism over $Y\backslash A$
such that $\cL$ possesses an extension as an invertible sheaf to $Z$.
\end{proposition}
\begin{proof}
Denote by $\nu:\wh Y \to Y$ the normalization of $Y$. The presheaf
$$
U\mapsto \{\sigma \in  (\nu_* \cO_{\wh Y})(U); \sigma|U\backslash A\in \cO_Y(U\backslash A)\}
$$
defines a coherent $\cO_Y$-module, the so-called  {\em gap sheaf}
$$
\cO_Y[A]_{\nu_*\cO_{\wh Y}}
$$
on $Y$ (cf.\ \cite[Proposition 2]{siu:gap}). It carries the structure of
an $\cO_Y$-algebra. According to Houzel \cite[Prop.\ 5 and Prop.\ 2]{hou}
it follows that it is an $\cO_Y$-algebra of finite presentation, and hence
its analytic spectrum provides a complex space $Z$ over $Y$ (cf.\ also
Forster \cite[Satz 1]{fo}).
\end{proof}
Finally, we have to deal with the question of a {\em global} extension of
positive currents. Siu's Thullen-Remmert-Stein type theorem \cite[Theorem
1]{siu:curr} gives an answer for currents on open sets in a complex number
space $\mathbb C^N$. An extension for $(1,1)$-currents exists and is
uniquely determined by one local extension into each irreducible
hypersurface component of the critical set. In a global situation we need
an extra argument, namely the fact that (under some assumption), one can
single out an extension, which is unique and does not depend upon the
choice of some local extension.

\begin{proposition}\label{pr:extcurr}
Let $A$ be a closed analytic subset of a normal complex space $Y$, and let
$Y'=Y\backslash A$.

Let $\omega'$ be a closed, positive current on $Y'$, with vanishing Lelong
numbers. Assume that for any point of $A$ there exists an open
neighborhood $U\subset Y$ such that $\omega'|U\cap Y'$ can be extended to
$U$ as a closed positive current. Then $\omega'$ can be extended to all of
$Y$ as a positive current.

If the local extensions of $\omega'$ have at most analytic singularities,
then also the constructed global extension has at most analytic
singularities.
\end{proposition}
\begin{proof}
We first assume that $Y$ is smooth and that $A$ is a simple normal
crossings divisor. Let $\omega_U$ be a (positive) extension of
$\omega'|U\cap Y'$.

We apply \eqref{eq:decomp}:
$$
\omega_U = \sum_{k=0}^\infty \mu_k [Z_k] + R\, .
$$
Since the Lelong numbers of $\omega'$ vanish everywhere, the sets $Z_k$
must be contained in $A$ so that the positive residual current $R$ is the
null extension of $\omega'|U\cap Y'$. We now take the currents of the form
$R$ as local extensions. If $W\subset Y$ is open, then the difference of
any two such extensions is a current of order zero, which is supported on
$A\cap W$. By \cite[Corollary III (2.14)]{demaillybook} it has to be a
current of integration supported on $A\cap W$, so it must be equal to
zero.

In the case, where $Y$ is not necessarily smooth but normal, we consider a
desingularization of the pair $(Y,A)$ like in the proof of
Theorem~\ref{th:extlinebdl} and use the local notation. By
\cite[Theorem~1]{siu:curr} we have unique local extensions of the
pull-back of $\omega'$ into the locus of codimension greater or equal to
two. The Lelong numbers are still equal to zero, and the previous argument
applies so that the pull-back of $\omega'$ extends to $Z$ as a positive
current $\wt \omega$. The push-forward of $\wt\omega$ solves the problem.
Since $\wt \omega$ differs from the given local extensions (pulled back to
$Z$) only by (locally defined) currents of integration with support in
$\mu^{-1}\tau^{-1}(A)$ the push-forward of $\wt \omega$ again has at most
analytic singularities, if the local extensions of $\omega'$ have this
property.
\end{proof}

\section{Degenerating families of canonically polarized
varieties}\label{se:degfam} In this section we want to show that in a
degenerating family the curvature of the relative canonical bundle can be
extended as a positive closed current.

Given a canonically polarized manifold $X$ of dimension $n$ together with
an $m$-canonical embedding $\Phi= \Phi_{mK_X}: X \hookrightarrow \mathbb
P_N$, the \fs metric $h_{FS}$ on the hyperplane section bundle
$\cO_{\mathbb P_N}(1)$ defines a volume form
\begin{equation}\label{eq:Omega0}
\Omega^0_X= (\sum_{i=0}^N |\Phi_i(z)|^2)^{1/m}
\end{equation}
on the manifold $X$, such that
$$
\omega^0_X:=-\text{Ric}(\Omega^0_X) = \frac{1}{m} \; \omega^{FS}|X,
$$
where $\omega^{FS}$ denotes the Fubini-Study form on $\mathbb P_N$.
According to Yau's theorem, $\omega^0_X$ can be deformed into a \ke metric
$\omega_X=\omega^0_X +  \ddb u$. It solves the equation \eqref{eq:ke},
which is equivalent to
\begin{equation}\label{eq:ke1}
\omega^n_X=(\omega^0_X+\ddb u)^n= e^u\Omega^0_X.
\end{equation}

We will need the $C^0$-estimates for the (uniquely determined)
$\cinf$-function $u$.

The deviation of $\omega^0_X$ from being \ke is given by the function
\begin{equation}\label{eq:defF}
F=\log\frac{\Omega^0_X}{(\omega^0_X)^n}
\end{equation}
so that \eqref{eq:ke1} is equivalent to
\begin{equation}\label{eq:MA}
\omega^n_X=(\omega^0_X+\ddb u)^n= e^{u+F}(\omega^0_X)^n.
\end{equation}
We will use the $C^0$-estimate for $u$ from \cite[Proposition
4.1]{cheng-yau} (cf.\ \cite{aub,kob}).

\bigskip

\noindent {\bf $C^0$-estimate.} \textit{Let $\Box^0$ denote the complex
Laplacian on functions with respect to $\omega^0_X$ (with non-negative
eigenvalues). Then
\begin{equation}\label{eq:uplusF}
 u+F \leq -\Box^0(u).
\end{equation}
In particular the function $u$ is bounded from above by $\sup(-F)$.}

\bigskip

Now we come to the {\em relative situation}. Let $\ol S \subset
\C^q=\{(s^1,\ldots,s^q)\}$ be a polydisk around the origin, and let $S=\ol
S \backslash V(s^1\cdot\ldots \cdot s^r)$ for some $r \leq q$. Let
\begin{equation}\label{eq:locfam}
\xymatrix{{\ol\cX}\ar[r]^-{\Phi}\ar[dr]_f & \mathbb P_N\times
\ol S\ar[d]^{{\rm{pr_2}}}
\\ & \ol S}
\end{equation}
be a proper, flat family, which is smooth over $S$, and $\Phi$ an
embedding.  We set $\cL=\Phi^{*}\cO_{\mathbb P_N\times \ol S}(1)$.

We assume that $\cL|\cX \simeq \cK_{\cX/S}^{\otimes m}$. Again for $s\in
S$ we equip the fibers $\cX_s$ with \ke forms, and we study the induced
relative \ke volume form. It defines a hermitian metric on the relative
canonical bundle $\cK_{\cX/S}$, whose curvature form was denoted by
$\omega_\cX$ (cf.\ Theorem~\ref{th:main}).

The following Proposition plays a key role.

\begin{proposition}\label{pr:extomloc}
Under the above assumptions, the form $\omega_\cX$ extends to $\ol \cX$ as
a closed, real, positive $(1,1)$-current.
\end{proposition}
\begin{proof}
Like in the absolute case, the \fs hermitian metric $h_{FS}$ on $\cL$
defines a relative volume form
$$
\Omega^0_{\cX/S} = h_{FS}^{-1/m}|\cX.
$$

In a similar way $\omega^0_{\cX/S}$ is defined:
$$
\omega^0_{\cX/S}:=-\text{Ric}_{\cX/S}(\Omega^0_{\cX/S}).
$$
We have, restricted to any fiber of $s\in S$,
$$
\omega^0_{\cX_s}= \frac{1}{m} \; \omega^{FS}|\cX_s.
$$
We denote by $u$ and $F$ the functions with parameter $s$, which were
defined for the absolute case in \eqref{eq:ke1} and \eqref{eq:defF} above.
Let $\sigma(s)= s^1\cdot\ldots\cdot s^r$. Singular fibers occur only,
where $\sigma$ vanishes.

Now both the initial relative volume form $\Omega^0_{\cX/S}$, and the
relative volume form $(\omega^0_{\cX/S})^n$ that is induced by the
relative Fubini-Study form, are given in terms of polynomials. On $S$ we
consider the function
$$
\sup\left(\left.\frac{(\omega^0_{\cX/S})^n}{\Omega^0_{\cX/S}}\right|{\cX_s}\right).
$$
It follows immediately that for some positive exponent $k$ we have
$$
| \sigma(s) |^{2k}
\sup\left(\left.\frac{(\omega^0_{\cX/S})^n}{\Omega^0_{\cX/S}}\right|{\cX_s}\right)
\leq c
$$
for $s\in S$, and some real constant $c$, i.e.\
$$
| \sigma(s) |^{2k} \sup\{ e^{-F(z)}; f(z)=s\} \leq c.
$$
(Here $\ol S$ may have to be replaced by a smaller neighborhood of $0\in
\ol S$.)

We denote by $\omega_{\cX/S}$ the relative \ke form. Again
$$
\omega^n_{\cX/S}= e^u\Omega^0_{\cX/S},
$$
and the fiberwise $C^0$-estimate for $u$ yields
\begin{equation}\label{eq:estu}
| \sigma(s) |^{2k} \cdot e^u \leq c
\end{equation}
on $\cX$.

Now the global form on $\cX$ constructed in Theorem~\ref{th:main} is
\begin{eqnarray*}
\omega_\cX &=& \ii \pt\ol\pt \log(\omega^n_{\cX/S})\\
&=&\ii \pt\ol\pt \log(e^u \cdot \Omega^0_{\cX/S})\\
&=& \ii \pt\ol\pt \log\left(e^u\cdot h_{FS}^{-1/m}|\cX\right)\\
&=& \ii \pt\ol\pt \log\left(| \sigma(s) |^{2k}\cdot e^u\cdot h_{FS}^{-1/m}|\cX\right).
\end{eqnarray*}
Observe that $\log h_{FS}$ is of class $\cinf$ on $\ol\cX$.

We know from Theorem~\ref{th:main} that $\log(| \sigma(s) |^{2k}\cdot
e^u\cdot h_{FS}^{-1/m})$ is plurisubharmonic on $\cX$, and by
\eqref{eq:estu} it is bounded from above, hence it possesses a
plurisubharmonic extension to $\ol\cX$ \cite{g-r}.
\end{proof}

\medskip

Now we apply Proposition~\ref{pr:extomloc} to families over Hilbert
schemes. We fix the Hilbert polynomial and consider the Hilbert scheme of
embedded flat proper morphisms. Denote by $\lsc{n}{\ol\cH}$ the
normalization and by $\lsc{n}\cH \subset \lsc{n}{\ol\cH}$ the locus of
smooth fibers.
\begin{equation}\label{eq:embfam}
\xymatrix{{\ol\cX}\ar[r]^-{\Phi}\ar[dr]_f & \mathbb P_N\times
\lsc{n}{\ol\cH}\ar[d]^{{\rm{pr_2}}}
\\ & \lsc{n}{\ol\cH}}
\end{equation}
Over $\lsc{n}{\cH}$ the fibers are canonically polarized, i.e.\
$\Phi^{*}\cO_{\mathbb P_N\times \lsc{n}{\ol\cH}}(1)|\cX \simeq
\cK_{\cX/\lsc{n}{\cH}}^{\otimes m}$.

\begin{theorem}\label{th:singext}
Let $(\cK_{\cX/\cH},h)$ be the relative canonical bundle on the total
space over the Hilbert scheme, where the hermitian metric $h$ is induced
by the \ke metrics on the fibers. Then the curvature form extends to the
total space $\ol\cX$ over the compact Hilbert scheme $\ol\cH$ as a
positive, closed current $\omega^{KE}_{\ol \cX}$ with at most analytic
singularities.
\end{theorem}
\begin{proof}
We observe that  $\ol\cH$ may be desingularized (such that the preimage of
$\ol\cH\backslash \cH$ is a simple normal crossings divisor): Since
$\omega_\cX$ is already defined on the (possibly singular) space $\cX$, it
will not be affected by the process of taking a pull-back and push-forward
again. After a desingularization we apply Proposition~\ref{pr:extomloc}
together with Proposition~\ref{pr:extcurr} and take the push-forward of
the resulting current.
\end{proof}

\section{Extension of the \wp form for canonically polarized varieties
to the compactified Hilbert scheme}\label{se:compHilb}

Now we consider extensions of the \wp form to the  Hilbert scheme
$\ol\cH$. We have the situation of \eqref{eq:embfam}.

We want to apply Theorem~\ref{th:singext} and consider the fiber integral
analogous to \eqref{eq:fibint_ke} for the map $ \ol\cX \to \ol \cH$.
\begin{equation}\label{eq:extwp}
\omega_{\cH}^{WP}=\int_{\cX/\cH} (\omega_\cX)^{n+1}.
\end{equation}

\begin{theorem}\label{pr:wp_pot}
The \wp form $\omega_{\cH}^{WP}$ on $\cH$ given by the fiber integral
\eqref{eq:extwp} can be extended to $\ol \cH$ as a positive $d$-closed
$(1,1)$-current.
\end{theorem}
\begin{proof}
The statement is about the pull-back of the \wp form  to a normalization,
so again we assume normality of the base space. We desingularize the pair
$(\lsc{n}{\ol\cH}, \lsc{n}{\ol\cH}\backslash \lsc{n}{\cH})$. We can apply
Proposition~\ref{pr:extcurr} and restrict ourselves to the local situation
\eqref{eq:locfam}. As we want to avoid wedge products of currents, we use
the proof of  Proposition~\ref{pr:extomloc}. On $\cX$ we have
\begin{gather*}
\omega_\cX:=2\pi c_1(\cK_{\cX/S},h) =
   \ddb\log (e^u\Omega^0_{\cX/S})= \omega^0_{\cX} + \ddb u,
\end{gather*}
where $\omega^0_\cX$ comes from the \fs form pulled back to $\cX$.

We observe that
$$
(\omega^0_\cX +\ddb u)^{n+1} = (\omega^0_\cX)^{n+1} + \sum_{j=0}^n\ddb u
(\omega^0_\cX)^j(\omega^0_\cX + \ddb u)^{n-j}.
$$
The first term $(\omega^0_\cX)^{n+1}$ can be treated directly. The form
\hfill \\ $\omega^0_{\ol\cX}= (1/m)\omega^{FS}|\ol\cX$ is the restriction
of a positive $\cinf$ form on the total space -- it was shown in
\cite[Lemme 3.4]{va1} that the fiber integral
$$
\int_{\ol\cX/\ol S}(\omega^0_{\ol\cX})^{n+1}
$$
exists and defines a $d$-closed real $(1,1)$-current (positive in the
sense of currents) on $\ol S$, which possesses a {\em continuous
$\pt\ol\pt$-potential} on $\ol S$. (This fact also follows from the main
theorem of \cite{yoshi}.)

For the second term we use the results of Section~\ref{se:degfam}: Again
we may replace $u$ by $\wt u= u + 2k \log|\sigma(s)|$, which does not
change the value of the fiber integral on the interior.

Near any point of $\ol S\backslash S$ the potentials $\wt u$ are bounded
from above uniformly with respect to the parameter of the base. We avoid
taking wedge products of currents and consider
\begin{gather*}
\int_{\cX/S} \ddb \Big( \wt u  \cdot(\omega^0_\cX)^j(\omega^0_\cX + \ddb \wt u )^{n-j}\Big)= \\
\qquad \ddb \left(\int_{\cX/S} \wt u \cdot (\omega^0_\cX)^j(\omega^0_\cX +
\ddb \wt u )^{n-j}\right).
\end{gather*}
Now the $\pt\ol\pt$-potential is given by integrals over the fibers:
$$
\int_{\cX_s} \wt u \cdot (\omega^0_\cX)^j(\omega^0_\cX +
\ddb \wt u )^{n-j}
$$
The functions $\wt u $ are known to be uniformly bounded from above,
whereas the integrals
$$
\int_{\cX_s} (\omega^0_\cX)^j(\omega^0_\cX +
\ddb \wt u )^{n-j}
$$
are constant as functions of $s$. This shows the boundedness from above of
the potential for the singular \wp metrics for families of \ke manifolds
of negative curvature. Finally, we invoke again
Proposition~\ref{pr:extcurr}.
\end{proof}

\section{Moduli of canonically polarized manifolds}\label{se:comp}
In this section we give a short analytic/differential geometric proof of
the quasi-projectivity of moduli spaces of canonically polarized manifolds
depending upon the variation of the \ke metrics on such manifolds.

\begin{theorem}\label{th:posbdlmod}
Let $\cM$ be a component of the moduli space of canonically polarized
manifolds. Then there exists a compactification $\ol\cM$ together with a
holomorphic line bundle $\ol\lambda$ equipped with a singular hermitian
metric $\ol h^Q$ with the following properties:
\begin{itemize}
  \item[(i)] The restriction of $\ol h^Q$ to $\cM$ is a hermitian
      metric of class $\cinf$ in the orbifold sense, whose curvature
      is strictly positive.
  \item[(ii)] The curvature form of $\ol h^Q$ on $\ol\cM$ is a
      (semi-)positive current.
  \item[(iii)] The metric $\ol h^Q$ has at most analytic
      singularities.
\end{itemize}
\end{theorem}
\begin{proof}
We know from Artin's theorem \cite{artin} that $\cM$ possesses a
compactification, which is a complex Moishezon space.

{\em Construction of an hermitian line bundle $\lambda_\cM$ on $\cM$.} We
fix a number $m$ so that for all polarized varieties belonging to $\cM$
the $m$-th power of the canonical line bundle determines an embedding into
a projective space (Matsusaka's Big Theorem).

For any smooth family $f:\cX \to S$, where $S$ denotes a reduced complex
space, in Section~\ref{se:bgs} we constructed the determinant line bundle
$$
\lambda_S:=\det f_! ((\cK^{\otimes m}_{\cX/S}-(\cK^{\otimes m}_{\cX/S})^{-1})^{n+1})
$$
together with a Quillen metric $h^Q_S$, whose curvature form is equal to
the generalized \wp form $\omega^{WP}_S$ up to a numerical factor. The
construction is functorial so that the line bundle, together with the
Quillen metric and the \wp form descend to a line bundle $\lambda_\cM$ and
a current $\omega^{WP}_\cM$ on the moduli space $\cM$ in the orbifold
sense. We also know that the order of the automorphism groups of the
fibers is bounded on $\cM$ so that a finite power of the above determinant
line bundle actually descends as a line bundle to $\cM$. The form
$\omega^{WP}_\cM$ can be interpreted as of class $\cinf$ in the orbifold
sense or as a positive current with vanishing Lelong numbers on $\cM$,
since it has local, continuous $\pt\ol\pt$-potentials.

\medskip

{\em Compactifications.} We denote the canonical map by $\nu:\cH \to \cM$
and chose a compactification $\ol \cM$. Since $\ol \cM$ is Moishezon, we
can eliminate the indeterminacy set of $\nu$: Let $\tau: \wt{\ol\cH} \to
\ol\cH$ be a modification such that $\nu$ extends to a map $\mu:
\wt{\ol\cH} \to \ol \cM:$
$$
\xymatrix{\wt{\ol\cH}  \ar[d]_{\tau }  \ar[dr]^\mu  &\\
\ol\cH  \ar @{} [d] |-\cup \ar @{-->} [r]   & \ol\cM  \ar @{} [d]|-\cup \\
\cH   \ar[r]^\nu & \cM  \; .}
$$
\medskip

{\em Extension of the \wp current to \funo{\ol\cM}.} The construction of
the \wp form on the Hilbert scheme \funo{\cH} is known to be functorial,
it descends to the (open) moduli space \funo{\cM} in the orbifold sense.
In Theorem~\ref{pr:wp_pot} we constructed an extension to \funo{\ol\cH} as
a positive, closed current. Here we use the theorem to construct the
extension \funo{\omega^{WP}_{\wt{\ol\cH}}} of \funo{\omega^{WP}} to
\funo{\wt{\ol\cH}}. We claim that restricted to \funo{\mu^{-1}(\cM)} this
current is equal to the pull-back of the (smooth) \wp form
\funo{\mu^*\omega^{WP}_\cM} as a current. This follows from the
construction in Proposition~\ref{pr:extcurr}, where the extension was
defined as a null-extension, and as such it is uniquely determined. Now
the push forward of \funo{\omega^{WP}_{\wt{\ol\cH}}} to \funo{\ol\cM}
under $\mu$ yields an extension of \funo{\omega^{WP}_\cM}. Over points of
\funo{\ol\cM\backslash \cM} the extension of the \wp form cannot be
controlled.

\medskip

{\em Reduction to the normalization.} By Proposition~\ref{pr:nonnormal} it
is sufficient consider normalizations.

\medskip

{\em Extension of the determinant line bundle.} We apply
Theorem~\ref{th:extlinebdl}.
\end{proof}
Now the criterion \cite[Theorem 6]{s-t} is applicable, which yields the
quasi-projectivity of the moduli space \cite{v,viebuch,v2,ko}.

\medskip

Finally, we address the finiteness of the volume of the moduli space
stated in Corollary~\ref{co:finwp}.

We know from Theorem~\ref{pr:wp_pot} together with
Proposition~\ref{pr:extcurr} that we have extensions of $\omega^{WP}$
resp.\ to the compactified space $\ol\cM$ as a closed, positive current.
This current defines the Chern class of a certain line bundle on the
compactified moduli space. Although the maximum exterior product of this
current need not exist as a current, the theory of non-pluripolar products
of globally defined currents of Boucksom, Eyssidieux, Guedj, and Zeriahi
from \cite{begz} applies. (Observe that both the orbifold space structure
and the singularities do not present any problem.) As a result we get that
the integral of the differentiable volume form over the interior $\cM$ is
finite.

\section{Curvature of $R^{n-p}f_*\Omega^p_{\cX/S}(\cK_{\cX/S}^{\otimes m})$
-- Statement of the theorem and Applications}\label{se:statement}

\subsection{Statement of the theorem}\label{ss:curv} We consider an effectively
parameterized family $\cX \to S$ of
canonically polarized manifolds, equipped with \ke metrics of constant
Ricci curvature $-1$. For any $m>0$ the direct image sheaves
$f_*\cK_{\cX/S}^{\otimes(m+1)}=f_*\Omega^n_{\cX/S}(\cK_{\cX/S}^{\otimes
m})$ are locally free. For values of $p$ other than $n$ we assume local
freeness of
$$
R^{n-p}f_*\Omega^p_{\cX/S}(\cK_{\cX/S}^{\otimes m}),
$$
i.e. $\dim_\C H^{n-p}(\cX_s , \Omega^p_{\cX_s}(\cK^{\otimes m}_{\cX_s}) )=
{\it const.}$ In particular the base change property (cohomological
flatness) holds. The assumptions of Section~\ref{ss:dolb} are satisfied so
that we can apply Lemma~\ref{le:dolb}. If necessary, we replace $S$ by a
(Stein) open subset, such that the direct image is actually free, and
denote by $\{\psi^1,\ldots,\psi^r\}\subset
R^{n-p}f_*\Omega^p_{\cX/S}(\cK_{\cX/S}^{\otimes m})(S)$ a basis of the
corresponding free $\cO_S$-module. At a given point $s\in S$ we denote by
$\{(\pt/\pt s_i)|_s; i=1,\ldots,M\}$ a basis of the complex tangent space
$T_sS$ of $S$ over $\C$, where the $s_i$ are holomorphic coordinate
functions of a minimal smooth ambient space $U\subset \C^M$.

Let $A_{i \ol\beta}^\alpha(z,s) \pt_\alpha dz^{\ol\beta}$ be a harmonic
\ks form. Then for  $s \in S$ the cup product together with the
contraction defines
\begin{footnotesize}
\begin{eqnarray}
A_{i \ol\beta}^\alpha \pt_\alpha dz^{\ol\beta}\cup \textvisiblespace :
\cA^{0,n-p}(\cX_s,\Omega^p_{\cX_s}(\cK_{\cX_s}^{\otimes m})) &\to&
\cA^{0,n-p+1}(\cX_s,\Omega^{p-1}_{\cX_s}(\cK_{\cX_s}^{\otimes m}))\label{eq:cup1}\\ A_{\ol \jmath
\alpha}^\ol\beta \pt_\ol\beta dz^{\alpha}\cup \textvisiblespace :
\cA^{0,n-p}(\cX_s,\Omega^p_{\cX_s}(\cK_{\cX_s}^{\otimes m})) &\to&
\cA^{0,n-p-1}(\cX_s,\Omega^{p+1}_{\cX_s}(\cK_{\cX_s}^{\otimes m})),\label{eq:cup2}
\end{eqnarray}
\end{footnotesize}
where $p>0$ in \eqref{eq:cup1} and $p<n$ in \eqref{eq:cup2}.

We will apply the above products to harmonic $(0,n-p)$-forms. In general
the results are not harmonic. We use the notation $\psi^{\ol\ell}:=
\ol{\psi^\ell}$ for sections $\psi^k$ (and a notation of similar type for
tensors on the fibers):

\begin{theorem}\label{th:curvgen}
The curvature tensor for $R^{n-p}f_*\Omega^p_{\cX/S}(\cK_{\cX/S}^{\otimes
m})$ is given by
\begin{eqnarray}\label{eq:curvgen}
R_{i\ol\jmath}^{\phantom{{i\ol\jmath}}\ol\ell k}(s)&=& m \int_{\cX_s}
\left( \Box + 1 \right)^{-1}(A_i\cdot A_\ol\jmath) \cdot(\psi^k \cdot
\psi^\ol\ell) g\/ dV\nonumber\\
&& \quad + m \int_{\cX_s} \left( \Box + m \right)^{-1} (A_i\cup\psi^k)
\cdot (A_\ol\jmath \cup \psi^\ol\ell) g\/ dV \\
&& \quad + m \int_{\cX_s} \left( \Box - m \right)^{-1}
(A_i\cup\psi^\ol\ell)\cdot (A_\ol\jmath \cup \psi^k) g\/ dV.
\nonumber
\end{eqnarray}
The only contribution in \eqref{eq:curvgen}, which may be negative,
originates from the harmonic parts in the third term. It equals
$$
- \int_{\cX_s} H(A_i\cup\psi^\ol\ell) \ol{H(A_j\cup\psi^\ol k)} g dV.
$$
\end{theorem}
Concerning the third term, the theorem contains the fact that the positive
eigenvalues of the Laplacian are larger than $m$. (For $p=0$ the second
term in \eqref{eq:curvgen} is equal to zero, and for $p=n$ the third one
does not occur.) One can verify that in case $p=0$ and $m=1$ the
right-hand side of \eqref{eq:curvgen} is identically zero as expected.

Theorem~\ref{th:curvgen} will be proved in Section~\ref{se:compcurv}.

The pointwise estimate \eqref{eq:hker2} of the resolvent kernel (cf.\ also
Proposition~\ref{pr:resol}) translates into an estimate of the curvature.
\begin{proposition}\label{pr:est1}
Let $f:\cX \to S$ be a family of canonically polarized manifolds, and $s
\in S$. Let a tangent vector of $S$ at $s$ be given by a harmonic \ks form
$A$ and let $\psi$ be a harmonic $(p,n-p)$-form on $\cX_s$ with values in
the $m$-canonical bundle. Then
\begin{equation}\label{eq:est1a}
R(A,\ol A, \psi,\ol\psi) \geq P_n(d(\cX_s)) \cdot \|A\|^2\cdot
\|\psi\|^2 + \|H(A\cup \psi)\|^2 - \|H(A\cup \ol\psi)\|^2.
\end{equation}
\end{proposition}

\begin{proof}[Proof of the Proposition] We apply \eqref{eq:curvgen}. The first term involves the
Laplacian on functions. It has already been treated. Concerning the second
and third term, we use the fact that on $\cA^{0,n-p \pm 1}(\cX_s,\Omega^{p
\mp 1}_{\cX_s}(\cK_{\cX_s}^{\otimes m}))$ according to
Lemma~\ref{le:boxdboxdb} of Section~\ref{se:compcurv} the equation
$\Box_\pt=\Box_\ol\pt$ holds and that moreover all non-zero eigenvalues
are larger than $m$ by the Claim in the proof of Theorem~\ref{th:curvgen}.
\end{proof}

For $p=n$ we obtain the following result.
\begin{corollary}\label{co:curvmcan}
For $f_*\cK_{\cX/S}^{\otimes(m+1)}$ the curvature equals
\begin{eqnarray}\label{eq:curvmcan}
R_{i\ol\jmath}^{\phantom{i\ol\jmath}\ol\ell k}(s)&=& m \int_{\cX_s}
\left( \Box + m \right)^{-1} (A_i\cup\psi^k)
\cdot (A_\ol\jmath \cup \psi^\ol\ell) g\/ dV\nonumber\\ && \quad +
m \int_{\cX_s} \left( \Box + 1 \right)^{-1}(A_i\cdot A_\ol\jmath)
\cdot(\psi^k \cdot \psi^\ol\ell)
 g\/ dV.
\end{eqnarray}
\end{corollary}

The operator $(\Box + m)^{-1}$ in the first term of \eqref{eq:curvmcan} is
positive on the respective tensors, the second term yields an estimate:
Let
\begin{equation}\label{eq:inpro}
H^{\ol\ell k} = \int_{\cX_s} \psi^k \cdot \psi^\ol\ell g\, dV.
\end{equation}
\begin{corollary}\label{co:curv1}
Let $s\in S$ be any point. Let $\xi^i_k\in \C$. Then
\begin{equation}
R_{i\ol\jmath}^{\phantom{i\ol\jmath}\ol\ell k}(s) \xi^i_k\ol{\xi^j_\ell}
\geq m \cdot P_n(d(\cX_s) )\cdot G_{i\ol\jmath}^{WP}\cdot H^{\ol\ell k}
\cdot \xi^i_k\ol{\xi^j_\ell}.
\end{equation}
In particular the curvature is strictly Nakano-positive with the above
estimate.
\end{corollary}

Next, we set $m=1$ and take a dual basis $\{\nu_i\}\subset R^p
f_*\Lambda^p\cT_{\cX/S}(S)$ of the $\{\psi^k\}$ and normal coordinates at
a given point $s_0\in S$. (Again local freeness is assumed.)

We have mappings dual to \eqref{eq:cup1} and \eqref{eq:cup2}:
\begin{eqnarray}
A_{i \ol\beta}^\alpha \pt_\alpha dz^{\ol\beta}\we \textvisiblespace :
\cA^{0,p}(\cX_s,\Lambda^p\cT_{\cX_s}) &\to&
\cA^{0,p+1}(\cX_s,\Lambda^{p+1}\cT_{\cX_s})\label{eq:we1}\\ A_{\ol \jmath
\alpha}^\ol\beta \pt_\ol\beta dz^{\alpha}\we \textvisiblespace :
\cA^{0,p}(\cX_s,\Lambda^p\cT_{\cX_s}) &\to&
\cA^{0,p-1}(\cX_s,\Lambda^{p-1}\cT_{\cX_s}).\label{eq:we2}
\end{eqnarray}
Here, in  \eqref{eq:we1} the wedge product stands for an exterior product,
whereas in \eqref{eq:we2} the wedge product denotes a contraction with
both indices of $A_i$. Again for $p=n$ in \eqref{eq:we1}, and for $p=0$ in
\eqref{eq:we2} the value of the wedge product is zero.

Observing that the role of conjugate and non-conjugate tensors is being
interchanged, the curvature can be computed from Theorem~\ref{th:curvgen}.
Again because of $S$-flatness of $\Omega^p_{\cX/S}(\cK_{\cX/S})$ and
$\Lambda^p\cT_{\cX/S}$ resp.\ it is sufficient to require that $\dim_\C
H^{n-p}(\cX_S, \Omega^p_{\cX_s}(\cK_{\cX_s}))$ is constant, which is
equivalent to $\dim_\C H^p(\cX_s, \Lambda^p\cT_{\cX_s})$ being constant.
\begin{theorem}\label{th:curvgendual}
The curvature of $R^pf_*\Lambda^p\cT_{\cX/S}$ equals
\begin{eqnarray}\label{eq:curvgendual}
R_{i\ol\jmath   k \ol\ell }(s)&=&- \int_{\cX_s}
\left( \Box + 1 \right)^{-1}(A_i\cdot A_\ol\jmath)
\cdot(\nu_k \cdot \nu_\ol\ell) g\/ dV\nonumber\\
&& \quad - \int_{\cX_s} \left( \Box + 1 \right)^{-1} (A_i\wedge\nu_\ol\ell)
\cdot (A_\ol\jmath \wedge \nu_k) g\/ dV \\
&& \quad -  \int_{\cX_s} \left( \Box - 1 \right)^{-1}
(A_i\wedge \nu_k)\cdot (A_\ol\jmath \wedge \nu_\ol\ell) g\/ dV.
\nonumber
\end{eqnarray}
The only possible positive contribution arises from
$$
\int_{\cX_s}H(A_i\wedge \nu_k) H(A_\ol\jmath \wedge \nu_\ol\ell) g\/ dV.
$$
\end{theorem}

We observe that in Theorem~\ref{th:curvgendual} for $n=1$ and $p=1$ the
third term in \eqref{eq:curvgendual} is not present, and we have the
formula for the classical \wp metric on \tei space of Riemann surfaces of
genus larger than one \cite{tr,wo}.

For $p=1$ we obtain the curvature for the generalized \wp metric from
\cite{sch:curv}, (cf.\ \cite{siu:canlift}). Again we can estimate the
curvature like in Proposition~\ref{pr:est1}.

For $p=0$, one can verify that the curvature tensor \eqref{eq:curvgendual}
is identically zero as expected.

The following case is of interest.

\begin{proposition}\label{pr:est2}
Let $f:\cX \to S$ be a family of canonically polarized manifolds and $s
\in S$. Let tangent vectors of $S$ at $s$ be given by harmonic \ks forms
$A, A_1,\ldots, A_p$  on $\cX_s$. Let $R$ denote the curvature tensor for
$R^pf_*\Lambda^p\cT_{\cX/S}$. Then we have in terms of the \wp norms:
\begin{gather}
R(A,\ol A, H(A_1\wedge \ldots \wedge A_p) ,\ol{H(A_1\wedge \ldots \wedge
A_p)}) \leq  \nonumber \hspace{4cm} \\ \quad -  P_{n}(\cX_s) \cdot
\|A\|^2\cdot \| H(A_1\wedge \ldots \wedge A_p)  \|^2 + \|H(A\wedge
A_1\wedge \ldots \wedge A_p)\|^2.\label{eq:est1b}
\end{gather}
\end{proposition}
\begin{proof}
Since the $A$ and $A_i$ are $\ol\pt$-closed forms, we have $H(A\wedge
H(A_1\wedge \ldots\wedge A_p))=H(A\wedge A_1 \wedge \ldots \wedge A_p)$.
\end{proof}

Next, we define  higher \ks maps defined on the symmetric powers of the
tangent bundle of the base. Let $S^p(R^1f_*\cT_{\cX/S})$ denote the $p$-th
symmetric tensor power for $p>0$. Then the natural morphism
$$
S^p(R^1f_*\cT_{\cX/S})\to R^pf_*\Lambda^p\cT_{\cX/S}
$$
together with the \ks morphism $\rho_S:\cT_S \to R^1f_*\cT_{\cX/S}$
induces the \ks morphism of order $p$
\begin{equation}\label{eq:rhop}
\rho^p_S: S^p \cT_{S} \to R^pf_*\Lambda^p \cT_{\cX/S}.
\end{equation}
After tensorizing with $\C(s)$ we have fiberwise the maps
$$
\rho^p_{S,s}: S^pT_{S,s} \to H^p(\cX_s,\Lambda^p \cT_{\cX_s})
$$
that send a symmetric power
$$
\frac{\pt}{\pt s^{i_1}}\otimes \ldots \otimes \frac{\pt}{\pt s^{i_p}}
$$
to the class of
$$
A_{i_1}\wedge \ldots \wedge A_{i_p}:=
A_{i_1\ol\beta_1}^{\alpha_1}\pt_{\alpha_1}dz^{\ol\beta_1}
\wedge \ldots \wedge A_{i_p\ol\beta_p}^{\alpha_p}\pt_{\alpha_p}dz^{\ol\beta_p}.
$$
\begin{definition}
Let the tangent vector $\pt/\pt s\in T_{S,s}$ correspond to the harmonic
\ks tensor $A_s$. Then the generalized \wp function of degree $p$ on the
tangent space is
\begin{gather}\label{eq:WPp}
\| \pt/\pt s \|^{WP}_p  := \|A_s\|_p := \|H( \underbrace{A_s \wedge \ldots
\wedge A_s}_p) \|^{1/p}\hspace{4cm}\notag \\
:=  \left( \int_{\cX_s} H(A_s \wedge \ldots \wedge A_s)\cdot \ol{H(A_s
\wedge \ldots \wedge  A_s)} g \, dV \right) ^{1/2p}
\end{gather}
\end{definition}
Obviously $\|\alpha\cdot A_s\|_p=|\alpha| \cdot \|A_s\|_p$ for all
$\alpha\in \C$. The triangle inequality is not being claimed.

From now on we assume that $S$ is a locally irreducible (reduced) space.

Next, let $C\to S$ be a smooth analytic curve, whose image in $S$ again is
a curve, and let $\cX_C \to C$ be the pull-back of the given family over
$S$. All $R^pf_*\Lambda^p\cT_{\cX_C/C}/\textit{torsion}$ are locally free,
maybe zero. If on some open subset of $C$ the value $\rho^p_{C,s}(\pt/\pt
s|_s) \in H^p(\cX_s, \Lambda^p\cT_{\cX_s})$ is different from zero, where
$\pt/\pt s\neq 0$ denotes a tangent vector of $C$ at $s$, then
$R^pf_*\Lambda^p\cT_{\cX_C/C}/\textit{torsion}$ is not zero, and on a
complement of a discrete set of $C$ the sheaf
$R^pf_*\Lambda^p\cT_{\cX_C/C}$ is locally free.

Given a family over a smooth analytic curve $C\to S$, the $p$-th \wp
function of tangent vectors defines a hermitian (pseudo) metric on the
curve, which we denote by $G_p$. On the complement $C'$ of a discrete
subset of $C$ we can estimate the curvature.
\begin{lemma}\label{le:curvgp}
For any analytic curve $C\to S$ the curvature $K_{G_p}$ of $G_p$, at
points $s\in C'$ with $G_p(s)\neq 0$ satisfies the following inequality,
where the tensors $A_s$ represent \ks classes of tangent vectors of $C'$
at $s$ with $A_s^p\neq0$.
\begin{eqnarray}\label{eq:curvgp}
K_{G_p}(s) &\leq& \frac{1}{p}\left(- P_n(d(\cX_s))\frac{\|A_s\|^2_1}{\|A_s\|^2_p}
 +  \frac{\|A_s\|^{2p+2}_{p+1}}{\|A_s\|^{2p+2}_p} \right).
\end{eqnarray}

The second summand vanishes identically for $p=n$.
\end{lemma}

\begin{proof}
Let $A^p=H(A_s\wedge\ldots\wedge A_s)$ be the {\em harmonic projection} of
the $p$-fold exterior product of $A_s$. Then the curvature tensor for
$R^pf_* \Lambda^p\cT_{\cX/C}$ satisfies
\begin{gather}\label{eq:curvGp}
R(\pt_s,\pt_\ol s, A^p, \ol{A^p}) \geq -\frac{\pt^2\log(G_p^p)}{\pt
s\ol{\pt s}} \cdot G_p^p\cdot \|A^p\|^2 =\\ \nonumber- p \frac{\pt^2
\log(G_p)}{\pt s\ol{\pt s}} \cdot \|A\|^{2p}_p = p \cdot G_p\cdot
K_{G_p}\cdot \|A\|^{2p}_p .
\end{gather}
Here $R(\pt_s,\pt_\ol s,\textvisiblespace,\textvisiblespace)$ is the
curvature form applied to the tangent vectors $\pt/\pt s$ and $\pt/\ol{\pt
s}$ resp. With respect to $G_p$, we identify $G_p=\|\pt/\pt
s\|^2_p=\|A_s\|^2_p$ (cf. \eqref{eq:WPp}) so that
$$
R(A_s,\ol{A_s},A^p, \ol{A^p}) \geq p \cdot K_{G_p}\cdot \|A\|^{2p+2}_p.
$$
Now the estimate of Proposition~\ref{pr:est2} implies
\begin{gather*}
K_{G_p} \leq \frac{1}{p}\left( -P_n(d(\cX_s)) \|A\|_1^2 \|A\|_p^{2p} +
\|A\|_{p+1}^{2(p+1)} \right)\Big/\|A\|_p^{2(p+1)}.
\end{gather*}
\end{proof}

\subsection{Hyperbolicity conjecture of Shafarevich}
In this section we describe a short proof of the following special case of
Shafarevich's hyperbolicity conjecture for canonically polarized varieties
\cite{b-v, keko,kk, kv1, kv2, ko, m, v-z, v-z2}.

\smallskip

{\bf Application.} {\it If a compact smooth curve $C$ parameterizes a
non-isotrivial family of canonically polarized manifolds, its genus must
be greater than one.}

For the proof, we will show that one of the metrics $G_p$ on $C$ has
negative curvature.

Let $f:\cX \to C$ denote a non-isotrivial family of canonically polarized
varieties over a smooth compact curve. Again we denote by $A_s$ the
harmonic \ks form that represents the \ks class of a tangent vector
$0\neq\pt/\pt_s\in T_{C,s}$, and by $A^q_s$ the harmonic representative of
the $q$-fold wedge product.  Let $p_0$ be the maximum number $p$ for which
$A^p_s\neq 0$ on some open subset, i.e.\ on the complement of a finite
subset of $C$.

\begin{proposition}\label{pr:nonisotr}
    Under the above assumptions, the locally defined functions $\log G_{p_0}$
are subharmonic, and there exists a number $c>0$ such that for all $s$ in
the complement of a finite set of $C$ and all $A_s\neq 0$
\begin{equation}\label{eq:curvgpX}
K_{G_{p_0}}(s) \leq - \frac{1}{p_0}P_n(d(\cX_s)) \frac{\|A_s\|_{1}^2}{\|A_s\|^2_{p_0}}\leq -c.
\end{equation}
\end{proposition}
\begin{proof}
We chose the complement $C'$ of a finite set in $C$ so that for all $s\in
C'$ we have $A^{p_0}_s\neq 0$ and also $A_s^{p_0+1}=0$ for $s\in C'$. On
$C'$ by Lemma~\ref{le:curvgp} we have the first inequality in
\eqref{eq:curvgpX}. We claim that there exist a number $c'>0$ such that
\begin{equation}\label{eq:estquot}
\|A_s\|_{p_0}/\|A_s\|_1 \leq c' \text{ for all } A_a\neq 0 \text{ and all } s \in C'.
\end{equation}
Since $C$ is compact, we need to show boundedness only near points $s_0
\in C\backslash C'$: Let a point $s_0$ be given by $s=0$, where $s$ is a
local holomorphic coordinate on $C$. Using a differentiable local
trivialization of the family, we find representatives $B_s$ of the \ks
classes (which depend in a $\cinf$ way on the parameter). The $p_0$-fold
wedge products $B_s\wedge\ldots\wedge B_s$ have bounded norm so that also
the norms of $A^{p_0}_s=H(B_s\wedge\ldots\wedge B_s)$ are bounded. Since
the given family is not isotrivial, the sheaf
$R^1f_*\cT_{\cX/C}/\textit{torsion}$ is locally free and not zero. We need
to consider only the case that $A_s\to 0 $ for $s\to 0$. In this case we
find some power $k>0$ so that $s^{-k}A_s$ converges towards some non-zero
element from $(R^1\! f_*\cT_{\cX/C}(C)/\textit{torsion})\otimes\C(s_0)$.
So $\inf_{s\to 0}\|s^{-k}A_s\|_1 > 0$. At the same time the norms
$\|s^{-k} A_s\|_{p_0}$ stay bounded from above. The factors $|s^{-k}|$
cancel out in $\|s^{-k}A_s\|_{p_0}/\|s^{-k}A_s\|_1$, which shows
\eqref{eq:estquot} implying \eqref{eq:curvgpX}. The first inequality in
\eqref{eq:curvgpX} implies that the (locally defined function) $\log
G_{p_0}$ is subharmonic on $C'$, the second implies that it is bounded on
$C$ so that $\log G_{p_0}$ is subharmonic on $C$ (being defined in terms
of local coordinate systems).
\end{proof}

In \cite[3.2]{demasantacruz} Demailly gives a proof of the Ahlfors lemma
(cf.\ also \cite{g-rz}) for singular hermitian metrics of negative
curvature  in the context of currents using an approximation argument. We
will need the following special case:

\begin{proposition}\label{pr:ahlschw}
Let $\gamma=\gamma(s) \ii ds\wedge \ol{ds}$, $\gamma(s)\geq 0$ be given on
an open disk $\Delta_R=\{|s|<R\}$, where $\log \gamma(s)$ is a subharmonic
function such that $\ddb(\log \gamma) \geq A\, \gamma$ in the sense of
currents for some $A>0$. Let $\rho$ denote the Poincaré metric on
$\Delta_R$. Then $\gamma\leq \rho/A$ holds.
\end{proposition}

\begin{proof}[Proof of the Application]
Let $\varphi: \Delta_R \to C$ be a non-constant holomorphic map and
$\gamma= \varphi^*(G_{p_0})$. Then the assumptions of
Proposition~\ref{pr:ahlschw} are satisfied by
Proposition~\ref{pr:nonisotr}. So the curve $C$ is hyperbolic.
\end{proof}

\begin{proposition}
Any relatively compact open subspace of the moduli space of canonically
polarized manifolds is Kobayashi hyperbolic in the orbifolds sense.
\end{proposition}

\begin{proof}
For any canonically polarized manifold we consider a local universal
deformation given by a holomorphic family $f:\cX \to S$. The moduli space
possesses an open covering by quotients of the form $S/\Gamma$, where
$\Gamma$ is a finite group. Over a relatively compact subspace of the
moduli space the diameter of the canonically polarized manifolds $\cX_s$
will be bounded. Since $S$ itself will be singular we consider the tangent
cone of $S$, which consists fiberwise of $1$-jets.  These are the tangent
vectors, which are induced by local analytic curves $C$ through the given
point $s_0\in S$. (cf.\ \cite[Chapter 2.3]{kobook}). Now the estimates of
Proposition~\ref{pr:nonisotr} hold also for locally defined analytic
curves. One can show that numbers $c>0$ satisfying the second inequality
from the statement of the Proposition can be chosen uniformly over any
given relatively compact subset of the moduli space, only depending upon
$p_0$ so that the smallest of these numbers yields Kobayashi
hyperbolicity.
\end{proof}

\subsection{Finsler metric on the moduli stack}

We indicate, how to construct a Finsler metric of negative holomorphic
curvature on the moduli stack. Different notions of a Finsler metric are
common. We do not assume the triangle inequality/convexity. Such metrics
are also called {\em pseudo-metrics} (cf.\ \cite{kobook}).
\begin{definition}\label{de:fins}
Let $Z$ be a reduced complex space and let $T_cZ$ be the fiber bundle
consisting of the tangent cones of 1-jets. An upper semi-continuous
function
$$
F:T_cZ \to [0,\infty)
$$
is called Finsler pseudo-metric (or pseudo-length function), if
$$
F(av)=|a|F(v) \text{ for all } a\in \C, v\in TZ.
$$
\end{definition}
The triangle inequality on the fibers is not required for the definition
of the holomorphic (or holomorphic sectional) curvature: The holomorphic
curvature of a Finsler metric at a certain point $p$ in the direction of a
tangent vector $v$ is the supremum of the curvatures of the pull-back of
the given Finsler metric to a holomorphic disk through $p$ and tangent to
$v$ (cf.\ \cite{abate-patrizio}). (For a hermitian metric, the holomorphic
curvature is known to be equal to the holomorphic sectional curvature.)
The functions $G_p$ define Finsler (pseudo) metrics. Furthermore, any
convex sum $G=\sum_j a_j G_j$, $a_j>0$ is upper semi-continuous and has
the property that $\log G$ restricted to a curve is subharmonic.

\begin{lemma}[cf.\ {\cite[Lemma~3]{sch:framas}}]\label{le:convsum}
Let $C$ be a complex curve and $G_j$ a collection of pseudo-metrics of
bounded curvature, whose sum has no common zero. Then the curvatures $K$
satisfy the following equation.
\begin{equation}\label{eq:curvest}
K_{\sum_{j=1}^k G_j} \leq \sum_{j=1}^k \frac{G_j^2}{(\sum_{i=1}^k G_i)^2} K_{G_j} .
\end{equation}
\end{lemma}

Like in the proof of Proposition~\ref{pr:nonisotr} one shows that $
\|A\|_{p+1}/\|A\|_{p}$ can be uniformly bounded with respect to $S$ for
$\|A\|_p\neq 0$. Using Lemma~\ref{le:convsum} and Lemma~\ref{le:curvgp} a
convex sum $G=\sum_p\alpha_pG_p$, $\alpha_p>0$ with negative holomorphic
curvature is constructed. The metric $G$ on curves defines an (upper
semi-continuous) Finsler metric that descends to the moduli space in the
orbifold sense.

\begin{proposition}\label{pr:exfins}
On any relatively compact subset of the moduli space of canonically
polarized manifolds there exists a Finsler orbifold metric, whose
holomorphic curvature is bounded from above by a negative constant.
\end{proposition}

\section{Computation of the curvature}\label{se:compcurv}
The components of the metric tensor for
$R^{n-p}f_*\Omega^p_{\cX/S}(\cK_{\cX/S}^{\otimes m})$ on the base space
$S$ are the integrals of inner products of the harmonic representatives of
cohomology classes. We know from Lemma~\ref{le:dolb} that these are the
restrictions of certain $\ol\pt$-closed differential forms on the total
space. When we compute derivatives with respect to the base of these fiber
integrals, we will apply Lie derivatives with respect to lifts of tangent
vectors in the sense of Section~\ref{sb:fibint}. Taking {\em horizontal}
lifts simplifies the computations. The Lie derivatives of these pointwise
inner products can be broken up, and Lie derivatives of these differential
forms with values in the $m$-canonical bundle have to be taken. The
derivatives are covariant derivatives with respect to the hermitian
structure on this line bundle. Since we are dealing with alternating forms
we may use covariant derivatives also with respect to the \ka structure on
the fibers, which further simplifies the computations.

Again, we will use the semi-colon notation for covariant derivatives and
use a $|$-symbol for ordinary derivatives, if necessary. Greek indices are
being used for fiber coordinates, Latin indices indicate the base
direction. Dealing with alternating forms, for instance of degree $(p,q)$,
extra coefficients of the form $1/p!q!$ are sometimes customary; these
play a role, when the coefficients of an alternating form are turned into
skew-symmetric tensors by taking the average. However, for the sake of a
halfway simple notation, we follow the better part of the literature and
leave these to the reader.

\subsection{Setup}
As above, we denote by $f:\cX \to S$ a smooth family of canonically
polarized manifolds and we pick up the notation from
Section~\ref{se:posi}. The fiber coordinates were denoted by $z^\alpha$
and the coordinates of the base by $s^i$. We set $\pt_i=\pt/\pt s^i$,
$\pt_\alpha=\pt/\pt z^\alpha$.

Again we have {\em horizontal lifts of tangent vectors and coordinate
vector fields on the base}
$$
v_i= \pt_i + a_i^\alpha  \pt_\alpha.
$$
As above we have the corresponding harmonic representatives
$$
A_i=A^\alpha_{i\ol\beta}\pt_\alpha dz^{\ol \beta}
$$
of the \ks classes $\rho(\pt_i|_{s_0})$.

For the computation of the curvature it is sufficient to treat the case
where $\dim S =1$. We set $s=s_1$ and $v_s=v_1$ etc. In this case we write
$s$ and $\ol s$ for the indices $1$ and $\ol 1$ so that
$$
v_s= \pt_s + a_s^\alpha \pt_\alpha
$$
etc.

Sections of \RP will be denoted by letters like $\psi$.
\begin{gather*}
\psi|_{\cX_s} = \psi_{\alpha_1,\ldots,\alpha_p,\ol\beta_{p+1},\ldots,\ol\beta_n}
dz^{\alpha_1}\wedge \ldots \wedge dz^{\alpha_p}\wedge dz^{\ol\beta_{p+1}}\we
\ldots \we dz^{\ol\beta_n} \\
= \psi_{A_p\ol B_{n-p}} dz^{A_p}\we dz^{\ol B_{n-p}} \hspace{5cm}
\end{gather*}
where $A_p=(\alpha_1,\ldots,\alpha_p)$ and $\ol B_{n-p}=(\ol\beta_{p+1},
\ldots,\ol\beta_n)$. The further component of $\psi$ is
$$
\psi_{\alpha_1,\ldots,\alpha_p,\ol\beta_{p+1},\ldots,\ol\beta_{n-1},\ol s}
dz^{\alpha_1}\wedge \ldots \wedge dz^{\alpha_p} \we dz^{\ol\beta_{p+1}}\we
\ldots \we dz^{\ol\beta_{n-1}}\we \ol{ds}.
$$
Now Lemma~\ref{le:dolb} implies
\begin{equation}
\psi_{\alpha_1,\ldots,\alpha_p,\ol \beta_{p+1}, \ldots, \ol\beta_n |\ol s}
= \sum_{j=p+1}^n (-1)^{n-j}
\psi_{\alpha_1,\ldots,\alpha_p,\ol \beta_{p+1}, \ldots, \wh{\ol\beta}_j, \ldots, \ol\beta_n ,\ol s|\ol\beta_j }.
\end{equation}
Since these are the coefficients of alternating forms, on the right-hand
side, we may also take the covariant derivatives with respect to the given
structure on the fibers
$$
\psi_{\alpha_1,\ldots,\alpha_p,\ol \beta_{p+1},
\ldots, \wh{\ol\beta}_j, \ldots, \ol\beta_n ,\ol s;\ol\beta_j}.
$$
\subsection{Cup Product}
We define the cup product of a differential form with values in the
relative holomorphic tangent bundle and an (line bundle valued)
differential form now in terms of local coordinates.
\begin{definition}\label{de:cup}
Let
$$
\mu= \mu^\sigma_{\alpha_1,\ldots,\alpha_p,\ol\beta_1,\ldots, \ol\beta_q}\pt_\sigma \,
dz^{\alpha_1}\we\ldots\we dz^{\alpha_p}\we dz^{\ol\beta_1}\we\ldots\we dz^{\ol\beta_q},
$$
and
$$
\nu= \nu_{\gamma_1,\ldots,\gamma_a,\ol\delta_1,\ldots,\ol\delta_b}
dz^{\gamma_1}\we\ldots\we dz^{\gamma_a}\we dz^{\ol\delta_1}
\we\ldots\we dz^\ol{\delta_b}
$$
Then
\begin{gather}\label{eq:cup}
\mu\cup\nu := \mu^\sigma_{\alpha_1,\ldots,\alpha_p,\ol\beta_1,\ldots,
\ol\beta_q}
\nu_{\sigma\gamma_2,\ldots,\gamma_a,\ol\delta_1,\ldots,\ol\delta_b}
dz^{\alpha_1}\we\ldots\we dz^{\alpha_p}  \\
\nonumber \hspace{2.7cm}\we dz^{\ol\beta_1}\we\ldots\we dz^{\ol\beta_q}\we
dz^{\gamma_2}\we\ldots\we dz^{\gamma_q}\we dz^{\ol\delta_1} \we\ldots\we
dz^\ol{\delta_b}
\end{gather}
\end{definition}

\subsection{Lie derivatives}
Let again the base be smooth, $\dim S=1$ with local coordinate $s$. Then
the induced metric on \RP is given by  \eqref{eq:inpro}, where the
pointwise inner product equals
$$
\psi^k \cdot \psi^\ol\ell g\, dV = (\ii)^n (-1)^{n(n-p)} \frac{1}{g^m} \psi^k \we \psi^\ol\ell,
$$
and where $1/g^m$ stands for the hermitian metric on the $m$-canonical
bundle on the fibers.
\begin{lemma}\label{le:Lieder}
$$
\frac{\pt}{\pt s} H^{\ol\ell k} = \int_{\cX_s}
L_v(\psi^k \cdot \psi^\ol\ell) g\, dV = \langle L_v \psi^k, \psi^\ell  \rangle + \langle  \psi^k, L_\ol v  \psi^\ell \rangle ,
$$
where $L_v$ denotes the Lie derivative with respect to the canonical lift
$v$ of the coordinate vector field $\pt/\pt s$.
\end{lemma}
\begin{proof}
Taking the Lie derivative is not type-preserving. We need the
$(1,1)$-component: $L_v(g_{\alpha\ol\beta})= \big[ \pt_s + a^\alpha_s
\pt_\alpha , g_{\alpha,\ol\beta} \big]_{\alpha\ol\beta} =
g_{\alpha\ol\beta|s} + a^\gamma_{s}g_{\alpha\ol\beta;\gamma}+
a^\gamma_{s;\alpha}g_{\gamma\ol\beta} =-
a_{s\ol\beta;\alpha}+a^\gamma_{s;\alpha}g_{\gamma\ol\beta}=0 $. So
$L_v(\det(g_{\alpha\ol\beta}))=0$.
\end{proof}
We have the type decomposition for $\psi=\psi^k$ or $\psi=\psi^\ell$
\begin{equation}
L_v\psi = L_v\psi' + L_v\psi'',
\end{equation}
where $L_v\psi'$ is of type $(p,n-p)$ and $L_v\psi''$ is of type
$(p-1,n-p+1)$. We have
\begin{eqnarray}
L_v\psi' &=&   \big[\pt_s + a^\alpha_s \pt_\alpha, \psi_{A_p\ol B_{n-p}} dz^{A_p } \we dz^{\ol B_{n-q}}\big]_{(p,n-p)}
\nonumber   \\
&=& (\psi_{;s} + a^\alpha_s \psi_{;\alpha} + \sum_{j=1}^p a^\alpha_{s;\alpha_j}
\psi_{
{\tiny\vtop{
\hbox{$\alpha_1,\ldots,\alpha,\ldots,\alpha_p\ol B_{n-p}\;$}\vskip-.8mm
\hbox{$\phantom{\alpha_1,\ldots,}{|\atop j} $}}}}) dz^{A_p}\we dz^{\ol B_{n-p}} \label{eq:lvprime} \\
L_v\psi'' &=&\big[\pt_s + a^\alpha_s \pt_\alpha, \psi_{A_p\ol B_{n-p}} dz^{A_p }\we dz^{\ol B_{n-q}}\big]_{(p-1,n-p+1)}
\nonumber   \\
&=& \sum^p_{j=1} A^\alpha_{s\ol\beta_j}
\psi_{ {\tiny\vtop{ \hbox{$\alpha_1,\ldots,\alpha,\ldots,\alpha_p\ol
B_{n-p}$\;}\vskip-.8mm \hbox{$\phantom{\alpha_1,\ldots,}{|\atop j} $}}}} \nonumber \\
&& \quad
\vtop{\hbox{$dz^{\alpha_1}\we\ldots\we dz^{\ol\beta_p}\we\ldots\we
dz^{\alpha_p} \we dz^{\ol\beta_{p+1}}\we\ldots\we
dz^{\ol\beta_n}$}\hbox{$\phantom{dz^{\alpha_1}\we\ldots\we \; }{|\atop j} $}} \label{eq:lvsecond}
\end{eqnarray}
We also note the values for the derivatives with respect to $\ol v$.
\begin{eqnarray}
L_\ol v\psi' &=&   \big[\pt_\ol s + a^\ol\beta_\ol s \pt_\ol\beta, \psi_{A_p\ol B_{n-p}}
dz^{A_p }\we dz^{\ol B_{n-q}}\big]_{(p,n-p)} \nonumber   \\
&=& (\psi_{;\ol s} + a^\ol\beta_\ol s \psi_{;\ol\beta} + \sum_{j=1}^p a^\ol\beta_{\ol s;\ol\beta_j}
\psi_{
{\tiny\vtop{
\hbox{$A_p \ol\beta_{p+1},\ldots,\ol\beta,\ldots,{\ol\beta}_n\;$}\vskip-.8mm
\hbox{$\phantom{A_p \ol\beta_{p+1},\ldots,}{|\atop j}$}}}})
dz^{A_p}\we dz^{\ol B_{n-p}} \label{eq:lvbprime} \\
L_\ol v\psi'' &=& \big[\pt_\ol s + a^\ol\beta_\ol s \pt_\ol\beta, \psi_{A_p\ol B_{n-p}} dz^{A_p }\we
dz^{\ol B_{n-q}}\big]_{(p+1,n-p-1)} \nonumber   \\
&=& \sum^n_{j=p+1} A^\ol\beta_{\ol s \alpha_{p+1}}
\psi_{\tiny\vtop{ \hbox{$\alpha_1,\ldots,\alpha_p,\ol\beta_{p+1},\ldots,\ol \beta,\ldots,\ol\beta_n\;$}\vskip-.8mm
\hbox{$\phantom{\alpha_1,\ldots,\alpha_p,\ol\beta_{p+1},\ldots,}{|\atop j}$}}} \nonumber \\ && \quad
\vtop{\hbox{$dz^{\alpha_1}\we\ldots\we dz^{\alpha_p}\we dz^{\ol\beta_1}\we\ldots\we dz^{\alpha_{p+1}}\we\ldots
\we dz^{\ol\beta_n} $}
\hbox{$\phantom{dz^{\alpha_1}\we\ldots\we dz^{\alpha_p}\we dz^{\ol\beta_1}\we\ldots\we dz}{|\atop j}$}} \label{eq:lvbsecond}
\end{eqnarray}
\begin{lemma}
\begin{eqnarray}
  L_v\psi'' &=& A_s \cup \psi \label{eq:2} \\
  L_\ol v\psi'' &=& (-1)^p A_\ol s \cup \psi \label{eq:3}
\end{eqnarray}
\end{lemma}
\begin{proof}[Proof of \eqref{eq:2}.]
By \eqref{eq:lvsecond} we have
\begin{gather*}
L_v\psi'' = \hspace{10cm}\\ =\sum^p_{j=1}
A^\alpha_{s\ol\beta_p}\psi_{\alpha_1,\ldots,
\wh\alpha_j,\ldots,\alpha_p,\alpha, \ol B_{n-p}} dz^{\alpha_1}\we
\ldots\we \wh{dz^{\alpha_j}}\we\ldots \we dz^{\alpha_p} \we
dz^{\ol\beta_p}\we
\ldots\we dz^{\ol\beta_n}\\
= (-1)^{p-1}\sum^p_{j=1}
A^\alpha_{s\ol\beta_p}\psi_{\alpha,\alpha_1,\ldots,\alpha_{p-1},\ol\beta_{p+1},\ldots,\ol\beta_{n}}
dz^{\alpha_1}\we\ldots\we\ldots \we dz^{\alpha_{p-1}} \we
dz^{\ol\beta_p}\we \ldots\we dz^{\ol\beta_n}.
\end{gather*}
\end{proof}
\begin{proof}[Proof of \eqref{eq:3}]
The claim follows in a similar way from \eqref{eq:lvbsecond}.
\end{proof}

The situation is not quite symmetric because of Lemma~\ref{le:dolb}, which
implies that the contraction of the global $(0,n-p)$-form $\psi$ with
values in $\Omega^p_{\cX/S}(\cK_{\cX/S}^{\otimes m})$ is well-defined.
Like in Definition~\ref{de:cup} we have a cup product on the total space
(restricted to the fibers).
\begin{eqnarray*}
\ol v \cup \psi &=& (\pt_\ol s + a^\ol\beta_\ol s \pt_\ol\beta)  \cup \psi \\
&=&\psi_{A_p,\ol s, \ol \beta_{p+1},\ldots,\ol\beta_{n-1}} +
a^\ol\beta_\ol s \psi_{A_p,\ol\beta,\ol\beta_{p+1},\ldots,\ol \beta_{n-1} }
\end{eqnarray*}
\begin{lemma}\label{le:lvpsi1}
\begin{equation}\label{eq:lvpsi1}
L_\ol v\psi'= (-1)^p\ol\pt (\ol v \cup \psi).
\end{equation}
\end{lemma}
\begin{proof}
The proof follows from the fact that, according to Lemma~\ref{le:dolb},
$\psi$ is given by a $\ol\pt$-closed $(0,n-p)$-form on the total space
$\cX$ with values in a certain holomorphic vector bundle.
\end{proof}

We will need that the forms $\psi$ on the fibers are {\em also harmonic
with respect to $\pt$} (which was defined as the connection of the line
bundle $\cK^{\otimes m}_{\cX/S}$). The curvature of $(\cK_{\cX/S},g^{-1})$
being equal to $-\omega_\cX$ implies the following Lemma.
\begin{lemma}\label{le:ddb}
\begin{equation}\label{eq:ddb}
 \ii [\ol\pt, \pt] = - m L_\cX,
\end{equation}
where $L_\cX$ denotes the multiplication with $\omega_{\cX}$. The
analogous formula holds on all fibers $\cX_s$.
\end{lemma}
Now:
\begin{lemma}\label{le:boxdboxdb}
The following equation holds on  $\cA^{(p,q)}(\cK^{\otimes m}_{\cX_s})$.
\begin{equation}
\Box_\pt = \Box_\ol\pt + m\cdot  (n-p-q) \cdot id.
\end{equation}
In particular, the harmonic forms $\psi \in \cA^{(p,n-p)}(\cK_{\cX_s})$
are also harmonic with respect to $\pt$.
\end{lemma}
\begin{proof}
We use the formulas
$$
\ii\, \ol\pt^* = [\Lambda,\pt] \text{\quad and \quad} -\ii \pt^* = [\Lambda,\ol \pt],
$$
where $\Lambda$ denotes the adjoint operator to $L$. Then
$$
\Box_\pt-\Box_\ol\pt = [\Lambda,\ii (\pt \ol\pt + \ol\pt \pt)]=[\Lambda,m\cdot \omega_\cX] =m\cdot(n-p-q)\cdot id.
$$
\end{proof}
Now we compute the curvature in the following way. Because of
\eqref{eq:lvpsi1}
$$
\langle \psi^k, L_\ol v (\psi^\ell)'  \rangle = 0
$$
holds for all $s\in S$ so that by Lemma~\ref{le:Lieder}
\begin{gather*}
\frac{\pt}{\pt s} H^{\ol\ell k} =
\langle L_v \psi^k, \psi^\ell  \rangle +
\langle  \psi^k,  L_\ol v \psi^\ell \rangle 
=\langle (L_v \psi^k)', \psi^\ell \rangle + \langle \psi^k, (L_\ol v
\psi^\ell)' \rangle \\
=\langle (L_v \psi^k)', \psi^\ell \rangle.
\end{gather*}

Later in the computation we will use normal coordinates (of the second
kind) at a given point $s_0\in S$. The condition $(\pt/\pt s)H^{\ol\ell
k}|_{s_0}=0$ for all $k,\ell$ means that for $s=s_0$ the harmonic
projection
\begin{equation}\label{eq:HLv}
H((L_v \psi^k)') =0
\end{equation}
vanishes for all $k$.

In order to compute the second order derivative of $H^{\ol\ell k}$ we
begin with
\begin{equation}\label{eq:dsH}
\frac{\pt}{\pt s} H^{\ol\ell k} = \langle L_v \psi^k, \psi^\ell \rangle.
\end{equation}
which contains both $(L_v \psi^k)'$ and $(L_v \psi^k)''$. Now
\begin{gather*}
\pt_\ol s\pt_s \langle \psi^k,\psi^\ell \rangle = \langle L_\ol v L_v
\psi^k , \psi^\ell \rangle +\langle L_v \psi^k,L_v\psi^\ell \rangle   \\ =
\langle L_{[\ol v,v]}\psi^k , \psi^\ell\rangle + \langle  L_vL_\ol
v\psi^k,\psi^\ell \rangle + \langle L_v \psi^k,L_v\psi^\ell \rangle \\ =
\langle L_{[\ol v,v]}\psi^k , \psi^\ell\rangle + \pt_s\langle L_\ol v
\psi^k, \psi^\ell\rangle -\langle L_\ol v\psi^k,L_\ol v\psi^\ell\rangle +
\langle L_v\psi^k,L_v\psi^\ell\rangle
\end{gather*}
We just saw that $\langle L_\ol v\psi^k , \psi^\ell \rangle \equiv 0$.
Hence for all $s\in S$
\begin{equation}\label{eq:Lolvv}
\pt_\ol s\pt_s \langle \psi^k,\psi^\ell \rangle = \langle L_{[\ol
v,v]}\psi^k , \psi^\ell\rangle -\langle L_\ol v\psi^k,L_\ol
v\psi^\ell\rangle + \langle L_v\psi^k,L_v\psi^\ell\rangle
\end{equation}
Since we are computing Lie-derivatives of $n$-forms (with values in some
line bundle) we obtain
$$
\langle L_v\psi^k,L_v\psi^\ell\rangle =  \langle (L_v\psi^k)',(L_v\psi^\ell)'\rangle
-  \langle (L_v\psi^k)'',(L_v\psi^\ell)''\rangle,
$$
and
$$
\langle L_\ol v\psi^k,L_\ol v\psi^\ell\rangle =  \langle (L_\ol v\psi^k)',(L_\ol v\psi^\ell)'\rangle
-  \langle (L_\ol v\psi^k)'',(L_\ol v\psi^\ell)''\rangle.
$$
\begin{lemma}\label{le:vvb}
Restricted to the fibers $\cX_s$ the following equation holds for $L_{[\ol
v, v]}$ applied to $\cK^{\otimes m}_{\cX/S}$-valued functions and
differential forms resp.
    \begin{equation}\label{eq:vvb}
      L_{[\ol v, v]} =
      \big[-\varphi^{;\alpha}\pt_\alpha + \varphi^{;\ol\beta}\pt_{\ol\beta},\; \textvisiblespace \;\big]
      - m\cdot \varphi \cdot id
    \end{equation}
\end{lemma}
\begin{proof}
We first compute the vector field $[\ol v, v]$ on the fibers:
\begin{gather*}
[\ol v, v]= [\pt_\ol s + a^\ol\beta_\ol s \pt_\ol\beta, \pt_{s} +
a^\alpha_{s} \pt_\alpha ]\hspace{5cm} \\= \left(\pt_\ol s (a^\alpha_ s) +
a^\ol\beta_\ol s a^\alpha_{s|\ol\beta}\right)\pt_\alpha - \left( \pt_s
(a^\ol\beta_\ol s) + a^\alpha_{s}a^\ol\beta_{\ol
s|\alpha}\right)\pt_\ol\beta \, .
\end{gather*}
Now
\begin{gather*}
\pt_\ol s (a^\alpha_ s) = -\pt_\ol s (g^{\ol\beta\alpha}g_{s\ol\beta}) =
g^{\ol\beta\sigma} g_{\sigma\ol s| \ol \tau}g^{\ol\tau
\alpha}g_{s\ol\beta} - g^{\ol\beta\alpha}g_{s\ol \beta|\ol s} \\
\hspace{5cm} = g^{\ol\beta\sigma}a_{\ol s
\sigma;\ol\tau}g^{\ol\tau\alpha}a_{s\ol\beta} - g^{\ol\beta\alpha} g_{s\ol
s; \ol\beta}.
\end{gather*}
Equation \eqref{eq:varphi} implies that the coefficient of $\pt_\alpha$ is
$-\varphi^{;\alpha}$. In the same way the coefficient of $\pt_\ol \beta$
is computed.

Next, we compute the contribution of the connection on $\cK^{\otimes
m}_{\cX/S}$ which we denote by $[\ol v, v]_{\cK^{\otimes m}_{\cX/S}}$. We
use \eqref{eq:ddb}:
\begin{gather*}
[\pt_\ol s + a^\ol\beta_\ol s \pt_\ol\beta, \pt_{s} + a^\alpha_{s}
\pt_\alpha]_{\cK^{\otimes m}_{\cX/S}}\hspace{5cm}  \\ =-m\left(g_{s\ol s}
+ a^\ol\beta_{\ol s} g_{s\ol\beta} + a^\alpha_s g_{\alpha\ol s} +
a^\ol\beta_\ol s a^\alpha_s g_{\alpha\ol\beta} \right) = -m \varphi.
\end{gather*}
\end{proof}

\begin{lemma}\label{le:lvvpsi}
\begin{equation}\label{eq:lvvpsi}
\langle L_{[\ol v, v]} \psi^k , \psi^\ell  \rangle = -m \langle \varphi \psi^k, \psi^\ell\rangle
= - m \int_{\cX_s} (\Box + 1)^{-1}(A_s \cdot A_\ol s) \psi^k \psi^\ol\ell \, g \, dV .
\end{equation}
\end{lemma}
\begin{proof}
The $\pt$-closedness of the $\psi^k$ can be read as
$$
\psi^k_{;\alpha} = \sum_{j=1}^p
\psi_{
{\tiny\vtop{
\hbox{$\alpha_1,\ldots,\alpha,\ldots,\alpha_p\ol B_{n-p};\alpha_j\;$}\vskip-.8mm
\hbox{$\phantom{\alpha_1,\ldots,}{|\atop j} $}}}}.
$$
Hence
\begin{eqnarray*}
[\varphi^{;\alpha}\pt_\alpha, \psi^k_{A_p\ol B_{n-p}}]'& = &
\varphi^{;\alpha}\psi_{;\alpha} +  \sum_{j=1}^p \varphi^{;\alpha}_{\;
;\alpha_j} \psi^k_{ {\tiny\vtop{
\hbox{$\alpha_1,\ldots,\alpha,\ldots,\alpha_p\ol B_{n-p}\;$}\vskip-.8mm
\hbox{$\phantom{\alpha_1,\ldots,}{|\atop j} $}}}}\\
& = & \sum_{j=1}^p \big( \varphi^{;\alpha} \psi^k_{ {\tiny\vtop{
\hbox{$\alpha_1,\ldots,\alpha,\ldots,\alpha_p\ol B_{n-p}\;$}\vskip-.8mm
\hbox{$\phantom{\alpha_1,\ldots,}{|\atop j} $}}}}\big)_{;\alpha_j}\\
&=& \pt\big( \varphi^{;\alpha}\pt _\alpha\cup \psi^k\big).
\end{eqnarray*}
Now
\begin{gather*}
\langle [\varphi^{;\alpha}\pt_\alpha, \psi^k_{A_p\ol
B_{n-p}}],\psi^\ell\rangle = \langle [\varphi^{;\alpha}\pt_\alpha,
\psi^k_{A_p\ol B_{n-p}}]',\psi^\ell\rangle \qquad \\ \qquad=\langle
\pt\big( \varphi^{;\alpha}\pt _\alpha\cup \psi^k\big),\psi^\ell \rangle =
\langle \varphi^{;\alpha}\pt _\alpha\cup \psi^k , \pt^* \psi^\ell \rangle
=0.
\end{gather*}
In the same way we get
$$
\langle [\varphi^{;\ol\beta}\pt_\ol\beta, \psi^k_{A_p\ol
B_{n-p}}],\psi^\ell\rangle =0,
$$
and, according to Lemma~\ref{le:vvb}, we are left with the desired term.
\end{proof}

\begin{proposition}\label{pr:baseq}
In view of \eqref{eq:cup1} and \eqref{eq:cup2} we have
\begin{eqnarray}
\ol\pt(L_v\psi^k)'&=&  \pt(A_s\cup \psi^k)\label{eq:0} \\
  \ol\pt^*(L_v\psi^k)'&=& 0 \label{eq:4} \\
  \pt^*(A_s\cup\psi^k) &=&0 \label{eq:5} \\
 \ol\pt^* (L_\ol v\psi^k)'&=&  \pt^* (A_\ol s \cup \psi^k) \label{eq:6}\\
  \ol\pt(L_\ol v\psi^k)'&=& 0 \label{eq:1} \\
  \ol\pt^*(A_\ol s \cup \psi^k) &=&0 \label{eq:7}
\end{eqnarray}
\end{proposition}
The proof of the above proposition is the technical part of this article
and will be given at the end.

{\it Proof of Theorem~\ref{th:curvgen}.} Again, we may set $i=j=s$ and use
normal coordinates at a given point $s_0\in S$.

We continue with \eqref{eq:Lolvv} and apply \eqref{eq:lvvpsi}. Let $G_\pt$
and $G_\ol\pt$ denote the Green's operators on the spaces of
differentiable $\cK_{\cX_s}$-valued $(p,q)$-forms on the fibers with
respect to $\Box_\pt$ and $\Box_\ol\pt$ resp. We know from
Lemma~\ref{le:boxdboxdb} that for $p+q=n$ the Green's operators $G_\pt$
and $G_\ol\pt$ coincide.

We compute $\langle (L_v\psi^k)', L_v\psi^\ell)'\rangle$: Since the
harmonic projection\\ $H ((L_v\psi^k)')=0$ vanishes for $s=s_0$, we have
\begin{gather*}
(L_v\psi^k)'= G_\ol\pt \Box_\ol\pt (L_v\psi^k)'= G_\ol\pt \ol\pt^*\ol\pt
(L_v\psi^k)' = \ol\pt^*G_\ol\pt \pt(A_s\cup\psi^k)
\end{gather*}
by \eqref{eq:4} and \eqref{eq:0}. The form $\ol\pt(L_v\psi^k)'=
\pt(A_s\cup \psi^k)$ is of type $(p,n-p+1)$ so that by
Lemma~\ref{le:boxdboxdb} on the space of such forms $G_\ol\pt=(\Box_\pt
+m)^{-1} $ holds.

Now
\begin{gather*}
\langle (L_v\psi^k)', (L_v\psi^\ell)'\rangle =\langle \ol\pt^*G_\ol\pt
\pt(A_s\cup\psi^k), (L_v\psi^\ell)' \rangle\\ = \langle G_\ol\pt
\pt(A_s\cup\psi^k), \pt (A_s \cup \psi^\ell) \rangle = \langle (\Box_\pt
+m)^{-1} \pt (A_s\cup \psi^k), \pt (A_s\cup \psi^\ell)\rangle\\
= \langle \pt^* (\Box_\pt +m)^{-1} \pt (A_s\cup \psi^k), A_s\cup
\psi^\ell\rangle.
\end{gather*}

Because of \eqref{eq:5}
\begin{gather*}
\langle (L_v\psi^k)', (L_v\psi^\ell)'\rangle = \langle (\Box_\pt
+m)^{-1}\Box_\pt (A_s \cup \psi^k) , A_s\cup \psi^\ell\rangle\\
= \langle A_s\cup \psi^k, A_s \cup\psi^\ell\rangle -m \langle (\Box+m
)^{-1}(A_s \cup \psi_k), A_s \cup \psi^\ell\rangle.
\end{gather*}
(For $(p-1, n-p+1)$-forms, we write $\Box=\Box_\pt=\Box_\ol\pt$.)
Altogether we have
\begin{equation}\label{eq:part2}
\langle L_v \psi^k , L_v \psi^\ell \rangle|_{s_0} = - m
\int_{\cX_s} (\Box +m)^{-1}(A_s\cup\psi^k)\cdot (A_\ol s\cup \psi^\ol\ell)\, g\, dV.
\end{equation}
Finally we need to compute $\langle L_\ol v\psi^k, L_\ol v \psi^\ell
\rangle$.

By equation \eqref{eq:3} we have  $(\langle L_\ol v\psi^k)'', (L_\ol v
\psi^\ell)'' \rangle = \langle A_\ol s \cup \psi^k   , A_\ol s \cup
\psi^\ell \rangle$. Now Lemma~\ref{le:lvpsi1} implies that the harmonic
projections of the $(L_\ol v \psi^k)'$ vanish for all parameters $s$. So
\begin{gather*}
\langle (L_\ol v \psi^k)',  (L_\ol v \psi^\ell)'\rangle = \langle G_\ol\pt
\Box_\ol\pt (L_\ol v \psi^k)',  (L_\ol v \psi^\ell)'\rangle \\
\vtop{\hbox{$=$}\vskip-4mm\hbox{\tiny$\!\! \eqref{eq:1}$} } \langle
(G_\ol\pt \ol\pt {\ol\pt}^* L_\ol v \psi^k)', (L_\ol v \psi^\ell)'\rangle
= \langle
(G_\ol\pt \ol\pt^* L_\ol v \psi^k)', \ol\pt^* (L_\ol v \psi^\ell)'\rangle\\
\vtop{\hbox{$=$}\vskip-4mm\hbox{\tiny$\!\! \eqref{eq:6}$} } \langle
G_\ol\pt \pt^*(A_\ol s \cup \psi^k), \pt^*(A_\ol s\cup \psi^\ell) \rangle.
\hspace{3cm}
\end{gather*}
Now the $(p, n-p-1)$-form $\ol\pt^*(L_\ol v\psi^k)'=  \pt^* (A_\ol s \cup
\psi^k)$ is orthogonal to both the spaces of $\ol\pt$- and $\pt$-harmonic
forms. On these, by Lemma~\ref{le:boxdboxdb} we have
$$
\Box_\ol\pt=\Box_\pt - m\cdot id.
$$
We see that all eigenvalues of $\Box_\pt$ are larger or equal to $m$ for
$(p,n-p-1)$-forms.

{\bf Claim.} {\it Let $\sum_\nu \lambda_\nu \rho_\nu$ be the eigenfunction
decomposition of $A_\ol s\cup \psi^k$. Then all $\lambda_\nu > m$ or
$\lambda_0=0$. In particular $(\Box - m)^{-1}(A_\ol s\cup \psi^k)$
exists.}

In order to verify the claim, we consider $\pt^*(A_\ol s\cup \psi^k)=
\sum_\nu \pt^*(\rho_\nu)$ with
$$
\Box_\pt \pt^*(\rho_\nu) = \lambda_\nu \pt^*(\rho_\nu) = \Box_\ol\pt
\pt^*(\rho_\nu) + m \cdot\pt^*(\rho_\nu).
$$
This fact implies that $\sum_\nu \pt^*(\rho_\nu)$ is the eigenfunction
expansion with respect to $\Box_\ol\pt$ and eigenvalues $\lambda_\nu
-m\geq 0$ of $\pt^*(A_\ol s\cup \psi^k)=\ol\pt^*(L_\ol v \psi^k)$. The
latter is orthogonal to the space of $\ol\pt$-harmonic functions so that
$\lambda_\nu-m=0$ does not occur. (The harmonic part of $A_\ol v\cup
\psi^k$ may be present though.) This shows the claim.

Now
$$
G_\ol\pt  \pt^*(A_\ol s \cup \psi^k) = (\Box_\pt -m)^{-1} \pt^*(A_\ol s \cup \psi^k)
$$
so that \eqref{eq:7} implies
\begin{gather*}
\langle (L_\ol v \psi^k)',  (L_\ol v \psi^\ell)'\rangle = \langle
(\Box_\pt -m)^{-1} \Box_\pt (A_\ol s \cup \psi^k) ,A_\ol s \cup \psi^\ell
\rangle \\
=  \langle A_\ol s \cup \psi^k ,A_\ol s \cup \psi^\ell \rangle + m\cdot
\langle (\Box_\pt -m)^{-1} (A_\ol s \cup \psi^k) ,A_\ol s \cup \psi^\ell
\rangle.
\end{gather*}
Now \eqref{eq:3} yields the final equation (again with $\Box_\ol\pt =
\Box_\pt = \Box$ for $(p+1,n-p-1)$-forms)
\begin{equation}\label{eq:part3}
\langle L_\ol v \psi^k,  L_\ol v \psi^\ell\rangle = m \int_{\cX_s} (\Box - m)^{-1}( A_\ol s \cup \psi^k)
\cdot (A_s \cup \psi^\ol\ell) \, g\, dV.
\end{equation}
The main theorem follows from \eqref{eq:lvvpsi}, \eqref{eq:part2},
\eqref{eq:Lolvv}, and \eqref{eq:part3}. \qed

\begin{proof}[Proof of Proposition~\ref{pr:baseq}]
We verify \eqref{eq:0}: We will need various identities. For simplicity,
we drop the superscript $k$. The tensors below are meant to be
coefficients of alternating forms on the fibers, i.e.\ skew-symmetrized.
\begin{equation}\label{eq:aux1}
\psi_{;s\ol\beta_{n+1}}= \psi_{;\ol\beta_{n+1}s} - m \cdot g_{s\ol\beta_{n+1}} \psi=
 m\cdot a_{s\ol\beta_{n+1}}\psi
\end{equation}
\begin{gather} \label{eq:aux2}
\psi_{;\alpha\ol\beta_{n+1}}= \psi_{;\ol\beta_{n+1}\alpha}- m\cdot
g_{\alpha\ol\beta_{n+1}}\psi \hspace{5cm} \\ \nonumber - \sum^p_{j=1}
\psi_{\tiny \vtop{ \hbox{$\alpha_1,\ldots,\sigma,\ldots,\alpha_p, \ol
B_{n-p}$}\vskip-1.5mm\hbox{$ \phantom{\alpha_1,\ldots,}{|\atop j}$}}
}R^\sigma_{\; \alpha_j\alpha\ol\beta_{n+1}}-\sum^n_{j=p+1}
\psi_{\tiny\vtop{\hbox{$A_p \ol\beta_{p+1},
\ldots,\ol\tau,\ldots,\ol\beta_{n}  $}\vskip-1.5mm\hbox{$\phantom{A_p
\ol\beta_{p+1}, \ldots,}{|\atop j}  $}}}
R^\ol\tau_{\:\ol\beta_j\alpha\ol\beta_{n+1} }
\end{gather}
\begin{gather}\label{eq:aux3}
a^\alpha_{s;\alpha_j\ol \beta_{n+1}} = A^\alpha_{s\ol\beta_{n+1};\alpha_j}
+ a^\sigma_s R^\alpha_{\; \sigma\alpha_j\ol\beta_{n+1}}
\end{gather}
Now, starting from \eqref{eq:lvprime} we get, using \eqref{eq:aux1},
\eqref{eq:aux2}, and \eqref{eq:aux3},
\begin{gather*}
\ol\pt L_v\psi' = \Big( \psi_{;s\ol\beta_{n+1}} +
A^\alpha_{s\ol\beta_{n+1}}\psi_{;\alpha} + a^\alpha_s
\psi_{;\alpha\ol\beta_{n+1}}  +
\sum^p_{j=1}a^\alpha_{s;\alpha_j\ol\beta_{n+1}}\psi_{\alpha_1,\ldots,\alpha,\ldots,\alpha_p,\ol
B_{n-p}}\\ \nonumber + \sum^p_{j=1} a^\alpha_{s;\alpha_j}
\psi_{\alpha_1,\ldots,\alpha,\ldots,\alpha_p,\ol B_{n-p};\ol\beta_{n+1}}
\Big) dz^{\ol\beta_{n+1}}\we dz^{A_p}\we dz^{\ol B_{n-p}}
\\ \nonumber 
= \Big(A^\alpha_{s\ol\beta_{n+1}}\psi_{;\alpha} + \sum^p_{j=1}
A^\alpha_{s\ol\beta_{n+1};\alpha_j}
\psi_{\alpha_1,\ldots,\alpha,\ldots,\alpha_p,\ol
B_{n-p}}\Big)dz^{\ol\beta_{n+1}}\we dz^{A_p}\we dz^{\ol B_{n-p}}.
\end{gather*}
Because of the fiberwise $\pt$-closedness of $\psi$ this equals
\begin{gather*}
\sum^p_{j=1} \big(A^\alpha_{s\ol\beta_{n+1}} \psi_{\tiny \vtop{
\hbox{$\alpha_1,\ldots,\alpha,\ldots,\alpha_p, \ol
B_{n-p}$}\vskip-1.5mm\hbox{$ \phantom{\alpha_1,\ldots,}{|\atop j}$}}  }
\big)_{;\alpha_j}
dz^{\ol\beta_{n+1}}\we dz^{A_p}\we dz^{\ol B_{n-p}} \\
= (-1)^n \sum^p_{j=1} \big(A^\alpha_{s\ol\beta_{n+1}}
\psi_{\alpha,\alpha_2,\ldots,\alpha_p, \ol B_{n-p}} \big)_{;\alpha_1}
dz^{\alpha_1}\we dz^{A_{p-1}}\we dz^{\ol\beta_1}\we\ldots\we
dz^{\ol\beta_{n+1}}\\ \nonumber
=\pt \Big((-1)^n A^\alpha_{s\ol\beta_{n+1}}
\psi_{\alpha,\alpha_2,\ldots,\alpha_p,\ol\beta_{p+1},\ldots,\ol\beta_n}
dz^{A_{p-1}}\we dz^{\ol B_{n+1}}\Big) = \pt\big( A_s \cup \psi\big).
\end{gather*}
This shows \eqref{eq:0}.

Next, we prove \eqref{eq:4}. We begin with \eqref{eq:lvbprime}. We first
note
\begin{equation*}\label{eq:dsGamma}
\pt_s(\Gamma^\sigma_{\alpha\gamma})= -a^\sigma_{s; \alpha\gamma}
\end{equation*}
which follows in a straightforward way. Now this equation implies
\begin{gather*}
\pt_s(\psi_{;\gamma}) = \psi_{;s\gamma} - \sum^p_{j=1}
a^\sigma_{s;\alpha_j\gamma}\psi_{\tiny\vtop{\hbox{$\alpha_1,
\ldots,\sigma,\ldots,\alpha_p \ol B_{n-p}
$}\vskip-1.5mm\hbox{$\phantom{\alpha_1, \ldots,}{|\atop j} $} }}
\end{gather*}
so that (with $g^{\ol\beta_n\gamma}\psi_{;\gamma}=0$ and $\pt_s
g^{\ol\beta_n\gamma}= g^{\ol\beta_n\sigma} a^\gamma_{s;\sigma} $)
\begin{gather}\label{eq:aux01}
g^{\ol\beta_n\gamma}\psi_{;s\gamma} =
-\psi_{;\gamma}  g^{\ol\beta_n\sigma} a^\gamma_{s;\sigma} +
\sum^p_{j=1}g^{\ol\beta_n\gamma}
a^\sigma_{s;\alpha_j\gamma}\psi_{\tiny\vtop{\hbox{$\alpha_1,
\ldots,\sigma,\ldots,\alpha_p \ol B_{n-p}
$}\vskip-1.5mm\hbox{$\phantom{\alpha_1, \ldots,}{|\atop j} $} }}
\end{gather}
follows. Next, since fiberwise $\psi$ is $\ol\pt^*$-closed,
\begin{gather}\label{eq:aux02}
g^{\ol\beta_n\gamma} (a^\alpha_s\psi_{;\alpha})_{;\gamma}=
g^{\ol\beta_n\gamma} a^\alpha_{s;\gamma}\psi_{;\alpha},
\end{gather}
and with the same argument
\begin{gather}\label{eq:aux03}
g^{\ol\beta_n\gamma} \big(\sum^p_{j=1}
a^\sigma_{s;\alpha_j}\psi_{\tiny\vtop{\hbox{$\alpha_1,
\ldots,\sigma,\ldots,\alpha_p \ol B_{n-p}
$}\vskip-1.5mm\hbox{$\phantom{\alpha_1, \ldots,}{|\atop j} $} }}
\big)_{;\gamma} = g^{\ol\beta_n\gamma} \sum^p_{j=1}
a^\sigma_{s;\alpha_j\gamma}\psi_{\tiny\vtop{\hbox{$\alpha_1,
\ldots,\sigma,\ldots,\alpha_p \ol B_{n-p}
$}\vskip-1.5mm\hbox{$\phantom{\alpha_1, \ldots,}{|\atop j} $} }}.
\end{gather}
Now $\ol\pt^*(L_v\psi')=0$ follows from \eqref{eq:aux01},
\eqref{eq:aux02}, and \eqref{eq:aux03}.

We come to the $\pt^*$-closedness \eqref{eq:5}  of $A_s\cup \psi$. We need
to show that
$$
\big(A^\alpha_{s\ol\beta_{n+1}}\psi_{\alpha,\alpha_2,\ldots,\alpha_p,\ol\beta_{p+1},\ldots,\beta_n}
\big)_{;\ol\delta}g^{\ol\delta\alpha_p}
$$
vanishes. Since $\pt^*\psi=0$ the above quantity equals
$$
A^\alpha_{s\ol\beta_{n+1};\ol\delta} \psi_{\alpha,\alpha_2,\ldots,\alpha_p,\ol B_{n-p}}g^{\ol\delta\alpha_p}.
$$
Because of the $\ol\pt$-closedness of $A_s$ this equals
$$
(A^{\alpha\alpha_p}_{s})_{;\ol \beta_{n+1}}\psi_{\alpha,\alpha_2,\ldots,\alpha_p,\ol B_{n-p}}.
$$
However,
$$
A^{\alpha\alpha_p}_s=A^{\alpha_p\alpha}_s
$$
whereas $\psi$ is skew-symmetric so that also this contribution vanishes.

The proof of \eqref{eq:6}, \eqref{eq:1}, and \eqref{eq:7} is similar, we
remark that \eqref{eq:1} also follows from Lemma~\ref{le:lvpsi1}.
\end{proof}

{\bf Acknowledgements.} This work was begun during a visit to Harvard
University. The author would like to thank Professor Yum-Tong Siu for his
cordial hospitality and many discussions about the hyperbolicity and
curvature. His thanks also go to Robert Berman, Bo Berndtsson, Indranil
Biswas, Jeff Cheeger, Jean-Pierre Demailly, Gordon Heier, Stefan Kebekus,
Janos Kollár, Sándor Kovács, Mihai Paun, Thomas Peternell, and Stefano
Trapani for discussions and remarks. Also, he would like to thank the
referees for comments and remarks.

He would  have liked to thank Eckart Viehweg also in writing for
discussing details of the analytic proof of the quasi-projectivity of
$\cM_{can}$ -- it is too late now.

\end{document}